\RequirePackage{fix-cm}
\documentclass[twocolumn]{svjour3}          
\smartqed  
\usepackage{graphicx}
 \usepackage{mathptmx}      
\usepackage{amssymb}
\usepackage{bm}
\usepackage{amsmath}
\usepackage{subfigure}
\usepackage[title]{appendix}%
\DeclareMathOperator{\Tr}{Tr}
\begin{document}

\title{Efficient pseudometrics for data-driven comparisons of nonlinear dynamical systems
}


\author{Bryan Glaz        
}


\institute{B.  Glaz \at
              Weapons Sciences Division,  DEVCOM Army Research Laboratory \\
             Aberdeen Proving Ground,  USA\\
              \email{bryan.j.glaz.civ@army.mil}           
}

\date{Received: date / Accepted: date}

\maketitle

\begin{abstract}
Computationally efficient solutions for pseudometrics quantifying deviation from topological conjugacy between dynamical systems are presented. Deviation from conjugacy is quantified in a Pareto optimal sense that accounts for spectral properties of Koopman operators as well as trajectory geometry. Theoretical justification is provided for computing such pseudometrics in Koopman eigenfunction space rather than observable space.  Furthermore, it is shown that theoretical consistency with topological conjugacy can be maintained when restricting the search for optimal transformations between systems to the unitary group. Therefore the pseudometrics are based on analytical solutions for unitary transformations in Koopman eigenfunction space. Geometric considerations for the deviation from conjugacy Pareto optimality problem are used to develop scalar pseudometrics that account for all possible optimal solutions given just two Pareto points. The approach is demonstrated on two example problems; the first being a simple benchmarking problem and the second an engineering example comparing the dynamics of morphological computation of biological nonlinear muscle actuators to simplified mad-made (including bioinspired) approaches.  The benefits of considering operator and trajectory geometry based dissimilarity measures in a unified and consistent formalism is demonstrated. Overall, the deviation from conjugacy pseudometrics provide practical advantages in terms of efficiency and scalability, while maintaining theoretical consistency.
\keywords{Data-driven dynamics \and Topological conjugacy \and Dissimilarity measures \and Koopman eigenfunctions}
\subclass{37M10 \and 37C15 \and 15B10 \and 15B57 }
\end{abstract}

\section{Introduction}
\label{intro}

Data-driven approaches are becoming increasingly prevalent for modeling, predicting, controlling, and designing complex nonlinear dynamical systems across a variety of scientific and technical fields \cite{brunton_book}. Data-driven approaches are especially effective for analyzing experimental observations of complex dynamical systems for which no accurate equations are available,  or analyses for which knowledge of empirical or first-principles governing nonlinear equations must be interrogated in order to understand the system's behavior. Quantitative comparisons (dissimilarity measures) between nonlinear dynamical systems based on data are particularly important for a variety of machine learning and optimization tasks.  As a result, it is desirable for data-driven dissimilarity measures to (a) be computationally efficient such that the approach is scalable to analyzing multiple combinations of systems \cite{williams_shape,lds_distance_opt_algorithm} and (b) respect the inherent features of the dynamical system that generated the data \cite{bollt_conj}.  The first attribute implies an `efficient' quantitative measure where `efficient' is widely accepted to mean it can be computed in polynomial time or better \cite{comp_complexity}.  The second desired attribute requires that the approach be anchored to an appropriate equivalence relationship for nonlinear dynamical systems so that systems that are equivalent return zero dissimilarity when compared.  Topological conjugacy \cite{wiggins}, which provides conditions on both trajectories and operator level dissimilarity, is the equivalence relationship explored here. 

There has been extensive work in comparing linear dynamical systems from both a trajectory geometry and operator perspective.  Efficient kernel based approaches were developed in \cite{bc_kernel,kernel_ext}. These approaches quantify dissimilarity based on trajectory geometry and follow the behavioral framework perspective to identifying dynamical systems \cite{behavior_ds}.  While the focus was on linear dynamical systems, an extension to nonlinear systems was briefly discussed in \cite{bc_kernel}.  The nonlinear systems formulation from \cite{bc_kernel} closely resembles a Koopman operator \cite{applied_koop} perspective. However, system or operator theoretic dissimilarity measures were not considered.  A systems oriented perspective to computing `distance' between linear dynamical systems was later taken in \cite{6247929,lds_distance_opt_algorithm}. It was recognized that restricting to the orthogonal (maximally compact) subgroup of real invertible matrices could reduce the computational expense of solving for optimal alignments between systems.  Topological conjugacy considerations were not explored in \cite{6247929,lds_distance_opt_algorithm}. Furthermore, generalizations to complex transformations and the relevance to dynamical systems was not studied.

Another trajectory focused approach is generalized shape metrics \cite{williams_shape}, which focused on scalability and computational efficiency.  It was also shown how generalized shape metrics relate to (and are often equivalent to) other statistical shape approaches such as canonical correlation analysis \cite{cca} and representational similarity analysis \cite{rsa}.  Similar to \cite{6247929,lds_distance_opt_algorithm}, the approach is based on the assumption of orthogonal transformations between data sets.  By restricting to orthogonal transformations it was shown in \cite{williams_shape} that the resulting pseudometrics for comparing systems satisfy subadditivity (i.e. triange inequality) \cite{metrics}.  This was considered an important feature for large scale machine learning applications where inability to satisfy the triangle inequality can lead to inconsitent results. Furthermore, the alignment transformations are solved using the orthogonal one-sided Procrustes solution \cite{schonemann} which can be computed in polynomial (cubic) time.  However, the generalized shape metric approach does not account for the inherent properties of the underlying dynamical system operator.

Beginning from the perspective that topological conjugacy is the foremost equivalence relationship in dynamical systems,   the homeomorphic defect approach to comparing nearly conjugate dynamical systems was introduced in \cite{bollt_conj}. The homeomorphic defect approach assumes one has found a transformation between two systems such that the trajectories match. This transformation function was referred to as a `commuter.' If the commuter function is a homeomorphism, then the systems are topologically conjugate. Therefore, dissimilarity was quantified according to the `defects' in the commuter function that prevented it from being a homeomorphism.  Scalable computation of the commuter functions from a data-driven context, as well as quantifying dissimilarity when a commuter function cannot be found were not explored.  

Koopman operator theory provides an attractive avenue for data-driven dissimilarity quantification between nonlinear dynamical systems since it has been shown to be well suited to capturing topological conjugacy \cite{applied_koop,mezic_koop_spectrum}. Theoretical development of pseudometrics based on Koopman operators was first explored in \cite{mezicandbanaszuk:2004} and \cite{MEZIC2016454} where it was shown that harmonic averages over a single observable function composed with a complete set of periodic functions could be used to quantify dissimilarity between systems.  This approach is especially well suited to experimental situations in which one may have access to limited measurements. More recently,  Koopman operator level attributes were considered in \cite{redman2022,redman2024} where the Wasserstein distance on the Koopman eigenvalues was used to quantify dissimilarity.  This approach was theoretically based on the insight that if two systems are topologically conjugate, then they will have the same Koopman eigenvalues and their Koopman eigenfunctions will be equivalent under transformation \cite{applied_koop}. The approach in \cite{redman2022,redman2024} was effective at identifying equivalence from data even when two systems are topologically conjugate under a nonlinear transformation. Applicability for comparisons of systems with unequal dimensionality was also discussed. Furthermore, computing the Wasserstein distance is solvable in polynomial time since it only requires solution to the linear assignment problem for a permutation matrix that optimally aligns the Koopman eigenvalues between two systems.  The question of how optimal alignment from a Wasserstein permutation perspective affects trajectory dissimilarity (and vice versa) remains open, particularly away from conjugacy.

Koopman operators served as the foundation for development of dynamic similarity analysis (DSA)\cite{ostrow_DSA}, which was introduced as an extension to generalized shape metrics. The focus of DSA was on accounting for system level, rather than trajectory level, dissimilarity while maintaining the advantages of upholding the triangle inequality. As a result, the analysis of \cite{williams_shape} was adapted to the two-sided Procrustes problem in \cite{ostrow_DSA} where it  was shown that orthogonal (or special orthogonal) transformations guarantee satisfaction of the triangle inequality. However, because the approach requires solution to a non-convex optimization problem, it is not solvable in polynomial time even if using fast optimization algorithms such as the one described in \cite{lds_distance_opt_algorithm}. It was shown in  \cite{ostrow_DSA} that DSA recovers the pseudometric from \cite{redman2022,redman2024} if the data-driven Koopman operators are normal. Otherwise,  as in \cite{6247929,lds_distance_opt_algorithm}, DSA cannot be guaranteed to recover zero dissimilarity for topologically conjugate systems in general due to the assumption of orthogonal transformations in observable space. 

Therefore the main contributions presented here are to develop computationally efficient pseudometrics that are theoretically consistent with topological conjugacy while accounting for implications of being away from conjugacy.  The approach is based on data-driven Koopman operators and accounts for both operator and trajectory dissimilarity measures within a unified framework. This is done by formulating \emph{deviation from conjugacy} pseudometrics based on solutions to a Pareto optimality problem that accounts for both Koopman eigenvalues and Koopman eigenfunction dissimilarity measures.  It is shown that unitary transformations of Koopman eigenfunctions and Koopman eigenvalues are needed to capture topological conjugacy, and that computationally efficient analytical solutions to the Pareto optimality problem can be obtained when considering such transformations.

Section \ref{sec_gp_koopj} provides a brief description of Koopman operator theory that serves as the foundation for the data-driven approach. Sections \ref{subsec_conj} and \ref{sec_metrics} contain the core novel contributions related to theoretical development of the deviation from conjugacy pseudometric, as well as analytical solutions that can be solved in polynomial (cubic) time.  Results are presented in Section \ref{results} illustrating the effectiveness of the approach on a simple benchmarking problem as well as an engineering example. Concluding remarks are provided in Section \ref{conc}.

\section{Data-Driven Koopman Operator Approximation}
\label{sec_gp_koopj}

In the pseudometric development in the next sections, it is required that one has constructed data-driven approximations to the Koopman operators for the dynamical systems that are to be compared.   A brief description of Koopman operator theory along with implementation choices relevant to the results presented in Section \ref{results} are provided. The reader is referred to  \cite{applied_koop,mezic_koop_spectrum} for detailed discussions of Koopman operator theory. 

For a nonlinear dynamical system in which data is obtained in discrete time, one has
\begin{equation}
\label{nlsys}
x_{n+1}=f(x_n), x \in \mathbb{C}^N, f: x \rightarrow x,
\end{equation}
with iteration number (time-step) given by $n$ and a potentially infinite dimensional vector of nonlinear observable functions $\Psi_n(x_n):  x \rightarrow \mathbb{C}^{\infty}$, the Koopman operator $\mathcal{K}$ evolves the observables forward according to
\begin{equation}
\label{koop_exact}
\Psi_{n+1}=\mathcal{K} \circ \Psi_n.
\end{equation}

Since $\mathcal{K}$ cannot be exactly derived in general and $\Psi$ may be infinite dimensional, data driven approximations $K \in \mathbb{C}^{N_{\Psi} \times N_{\Psi}}$ can be computed on a finite set of observable functions $\Psi_n(x_n):  x \rightarrow \mathbb{C}^{N_{\Psi}}$. Dynamic mode decomposition(DMD) and extended dynamic mode decomposition(EDMD) are commonly used approaches\cite{dmd,edmd,rowleyetal:2009} to identify $K$ from observable vector data.  Given the data-driven Koopman operator approximation, the observables evolve as
\begin{equation}
\label{koop_approx}
\Psi_{n+1} \approx K  \Psi_n, n=1 \ldots N_{\Delta t},
\end{equation}
Furthermore, the Koopman eigenfunctions, $\Phi(x)$, can be expanded in terms of the finite set of Koopman observable functions such such that
\begin{equation}
\label{koop_approx_efs}
\Phi(x) \approx W \Psi(x), \; \Phi \in \mathbb{C}^{N_{\Psi} \times N_{\Delta t}}
\end{equation}
where $W$ is a matrix in which each row corresponds to an eigenfunction associated with one of the Koopman eigenvaues. The matrix $W$ is computed by taking the transpose of the left-eigenvectors of $K$ \cite{dmd,edmd}. It is assumed here that $W$ is invertible and $K$ is diagonalizable.

\subsection{Observable vector}
\label{koop_obs}
The Koopman observable functions are chosen by the user. Typically, they consist of functions of physical or application relevance, potentially including system states. Assuming the data is in discretized form, the Koopman observable vector used in this study is $\Psi_n = \left[\begin{matrix} 1 & \Psi_{\mathrm{primary},n}^T & \Psi_{\mathrm{aux},n} ^T \end{matrix}\right]$.  The primary observables $\Psi_{\mathrm{primary},n}$ denoting quantities of interest are chosen here to be the system states, control inputs $u$, and certain nonlinear observable functions $\eta(x,u)$,
\begin{equation}
\label{psi_prime}
\Psi_\mathrm{primary}^T = \left[\begin{matrix} x_1 & \ldots & x_N & u_1 & \ldots & u_{N_u} &  \eta_1 & \ldots & \eta_{N_\eta}\end{matrix}\right].
\end{equation}

Similar to EDMD, the auxiliary observables are used to accurately approximate the nonlinear behavior in the primary observables.  Thus they are a rich set of basis functions for capturing nonlinear behavior. They must be included in identification of the Koopman operator, but may not be of inherent interest to a user. The auxiliary Koopman observables used here are based on the correlation functions given by
\begin{equation}
\label{psi_aux}
\rho(i,j)= \prod_{k=1}^4 \max \left[0,1 -  \theta_k \vert \Psi_{\mathrm{primary},i}^{(k)} - \Psi_{\mathrm{primary},j}^{(k)}\vert \right],
\end{equation} 
where the $(i,j)$ indices denote the discretized time step in the data and the index $k$ denotes the component of the $\Psi_{\mathrm{primary}}$ vector.  These correlation functions were chosen because they were found to provide accurate approximations while avoiding ill-conditioned correlation matrices.  The choice of correlation function is user defined and a variety of options exist in the machine learning literature. The vector of auxiliary Koopman observables for time step $n$ is \\$\Psi_{\mathrm{aux},n} = \left[\begin{matrix} \rho(n,1) & \rho(n,2) & \ldots & \rho(n,N_{\Delta t}) \end{matrix}\right]^T$.  The parameters $\theta_k$ were found by training on data for time steps $> N_{\Delta t}$ to ensure that the identified Koopman models accurately predict for long periods of time.  Finally, the constant 1 is used for the first component in $\Psi$ to account for constant trends (e.g. mean values). From Eqs. \ref{psi_prime} and \ref{psi_aux}, $N_{\Psi} = 1+N+N_u+N_{g}+N_{\Delta t}$.

\subsection{Identification of the Koopman operator}

Given $\Psi_{n+1}$ and $\Psi_{n}$, $K$ is typically identified by the least squares approach \cite{dmd,edmd}.
Alternativey one could take a Gaussian process machine learning perspective to identify $K$ from the data \cite{NIPS2007_66368270,GP_ML_book}.  For example, see \cite{masuda2019,LIAN2020449} for applications to to identification of Koopman operators. The Gaussian process interpolation approach taken here results in each row of $K$ being a generalized least squares solution to the data, often referred to as a kriging model \cite{jon01}.  Therefore, the predictors for $\Psi_{n+1}$ as functions of $\Psi_{n}$ are the best linear unbiased predictors \cite{jon01}.  It was found for the engineering example presented in Section \ref{results} that the kriging and DMD/EDMD approaches led to similar accuracy but the Gaussian process based Koopman operators had lower norms which can be useful for stability of long term predictions \cite{stable_koop}. Thus the generalized least squares/kriging interpolation approach was used in this study to identify $K$ from $\Psi_{n+1}$ and $\Psi_{n}$ data.  The theory and analytical solutions described next in Sections \ref{subsec_conj} and \ref{sec_metrics} apply regardless of one's preferred approach for identifying $K$. 

\section{Topological Conjugacy from the Perspective of Koopman Operators}
\label{subsec_conj}

The conceptual foundation for comparing/classifying systems is that pseudometrics should return zero when two systems are equivalent. For dynamical systems, topological conjugacy serves as the theoretical foundation for equivalence \cite{wiggins}.  If two dynamical systems are topologically conjugate to another, then they are equivalent.  Building off the notion of equivalence, one can quantify dissimilarity between dynamical systems by their deviation from conjugacy, i.e.  if the deviation from conjugacy between system $a$ and $b$ is less than the deviation between system $b$ and $c$, then among systems $a$ and $c$, we can quantitatively say that $a$ is more similar to $b$.  A brief description of topological conjugacy and the connection with Koopman based data-driven approaches follows. The deviation from conjugacy pseudometrics are developed in Section \ref{sec_metrics}.

\begin{definition}
Two systems $x_{n+1}=f(x_n), x \in X$ and $y_{n+1}=g(y_n), y \in Y$ are topologically conjugate if there exists a homeomorphism $h: X \rightarrow Y, $ such that \cite{wiggins}  
\begin{equation}
\label{conj}
\begin{split}
& y=h(x), \\
& h \circ f(x) = g \circ h(x).
\end{split}
\end{equation}
\end{definition}
It can be shown by using the spectral equivalence of topologically conjugate systems proposition from \cite{applied_koop},  that if $f$ and $g$ are topologically conjugate and their Koopman operators $\mathcal{K}_{f}$and $\mathcal{K}_{g}$ have discrete spectra with eigenvalues $\lambda_{f,i}, \forall i=1, \dots, N_{\Psi}$, and $\lambda_{g,i}, \forall i=1, \dots, N_{\Psi}$ respectively then 
\begin{equation}
\label{eig_conj}
 \lambda_{f,i} = \lambda_{g,i}, \forall i=1, \dots, N_{\Psi}
\end{equation}
assuming an appropriate ordering of eigenvalues.
Although exact identification of Koopman operators $\mathcal{K}_{f}$ and $\mathcal{K}_{g}$ is impractical in general, one still benefits greatly from Eq.~\ref{eig_conj}. Assuming one has identified data-driven operators $K_f$ and  $K_g$ that accurateley approximate $\mathcal{K}_{f}$ and $\mathcal{K}_{g}$ with respect to the attractors sampled by the data, then Eq.~\ref{eig_conj} implies that the eigenvalues of $K_f$ and  $K_g$ will be equivalent for topologicaly conjugate systems.

In Koopman observable space, $\Psi_{f,n+1}=K_f \Psi_{f,n}$ and \\$\Psi_{g,n+1}=K_g \Psi_{g,n}$.Building upon the spectral equivalence of $K_f$ and  $K_g$ when the two systems are topologically conjugate and taking the homeomorphism $h$ to be an invertible transformation matrix $T \in GL(N_{\Psi})$, Eq.~\ref{conj} becomes:
\begin{equation}
\label{sys_conj}
\left\Vert K_f - T^{-1} K_g T \right\Vert_F = 0,  T \in GL(N_{\Psi})
\end{equation}
\begin{equation}
\label{traj_conj}
\left\Vert \Psi_g - T \Psi_f \right\Vert_F = 0, \Psi_{f,g} \in \mathbb{C}^{N_{\Psi} \times N_{\Delta t}},
\end{equation}
where $\Vert \; \Vert_F$ is the Frobenius norm.  Note that translation can be neglected since affine transformations cannot lead to similarity transformations between two linear systems; i.e.  $K_f \ne h^{-1} K_g h$ when $h \circ \Psi_f = T\Psi_f+b$.

Equations \ref{sys_conj} and~\ref{traj_conj} represent the equivalence conditions that any pseudometric consistent with topological conjugacy should recover if two systems are indeed conjugate. However, before proceeding with the derivation of such pseudometrics, it is first beneficial to recast the conditions in Eqs. \ref{sys_conj} and~\ref{traj_conj} associated with Koopman observable space, $\Psi$-space, into equivalent conditions written in Koopman eigenfunction space, $\Phi$-space. As will be discussed next and in Section~\ref{sec_metrics}, the pseudometrics in $\Phi$-space hold several advantages compared to solving for optimal transformations based on Eqs. \ref{sys_conj} and~\ref{traj_conj}.

\begin{proposition}
\label{psi_vs_phi}
If two discrete spectra systems $f$ and $g$ are topologically conjugate then 
\begin{equation}
\label{sys_conj_eig}
\left\Vert \Lambda_f - C^{-1} \Lambda_g C \right\Vert_F = 0
\end{equation} is an equivalent condition to Eq.~\ref{sys_conj} and 
\begin{equation}
\label{traj_conj_eig}
\left\Vert \Phi_g - C \Phi_f \right\Vert_F = 0
\end{equation} 
is an equivalent condition to Eq.~\ref{traj_conj},  where
\begin{equation}
C = W_g T W^{-1}_f, C \in GL(N_{\Psi}),
\end{equation}
the diagonal matrices $\Lambda_{f,g} \in \mathbb{C}^{N_{\Psi} \times N_{\Psi}}$  contain the arbitrarily ordered eigenvalues of $K_{f,g}$, and $W_{f,g} \in GL(N_{\Psi})$ contain the left-egenvectors of $K_{f,g}$ such that the $i\mathrm{th}$ row of $W_{f,g}$ corresponds to the eigenfunction with eigenvalue in element $(i,i)$ of $\Lambda_{f,g}$.
\end{proposition}

\begin{proof}
Using the composition property of groups, the transformation matrix $T$ can be written as the composition of other matrices from $GL(N_{\Psi})$ without loss of generality. Since $W_f$ and $W_g$ are assumed to be invertible,
\begin{equation} 
\label{TtoC}
T = W^{-1}_g C W_f.
\end{equation}
Substituting Eq.~\ref{TtoC} into Eq.~\ref{sys_conj} and recognizing that by definition $K_{f,g} =W^{-1}_{f,g} \Lambda_{f,g} W_{f,g}$ leads to
\begin{equation}
\label{K_Lambda}
\left\Vert K_f - T^{-1} K_g T \right\Vert_F = \left\Vert W^{-1}_f\left(\Lambda_f - C^{-1} \Lambda_g C \right)W_f\right\Vert_F.
\end{equation}
Since the systems are conjugate, Eq.~\ref{K_Lambda} equals zero which can only be true if $\left\Vert \Lambda_f - C^{-1} \Lambda_g C \right\Vert_F = 0$. Thus, at conjugacy, Eq. \ref{sys_conj} and \ref{sys_conj_eig} must both be true and are equivaent conditions.

Now consider the trajectory conjugacy condition.  From Eq. \ref{koop_approx_efs}, the observables can be written in terms of Koopman eigenfunctions as 
\begin{equation}
\label{obs_to_ef}
\Psi_{f,g} = W^{-1}_{f,g} \Phi_{f,g}.
\end{equation}
Subsitituting Eqs. \ref{TtoC} and \ref{obs_to_ef} into Eq. \ref{traj_conj} directly leads to the condition given by Eq. \ref{traj_conj_eig}. Thus Eq. \ref{traj_conj} and Eq. \ref{traj_conj_eig} are equivalent conditions at conjugacy.

\end{proof}

\begin{remark}
\label{remark_psi_vs_phi}
The equivalence between topological conjugacy conditions in $\Psi$-space and $\Phi$-space established in Proposition~\ref{psi_vs_phi} provides the conceptual foundation for the development of the pseudometrics in $\Phi$-space. This allows for computing deviations from conjugacy based on $C$ rather than $T$. As will be shown, computationally efficient and theoretically consistent analytical solutions for $C$ under certain group restrictions can be found whereas this is not the case if attempting to directly solve for $T$ in $\Psi$-space.  
\end{remark}

%
%
%

\section{Deviation from Conjugacy Pseudometrics}
\label{sec_metrics}
The development of deviation from conjugacy pseudometrics for comparing dynamical systems is described in this section.  Conceptual perspective is provided first, followed by simplifications to the Pareto optimality problem that needs to be solved.  Given these simplifications, computationally efficient analytical solutions in $\Phi$-space are provided, followed by their corresponding $\Psi$-space representations.

\subsection{Deviation from conjugacy}
\label{dfc_concept}
Motivated conceptually by the homeomorphic defect approach \cite{bollt_conj},  dissimilarity between dynamical systems can be accounted for by quantifying how much the systems deviate from topological conjugacy. The perspective in \cite{bollt_conj} was based on having a ``commuter" $h$ that satisfies Eq. \ref{conj} and then quantifying dissimilarity in terms of how $h$ differs from being a homeomoprhism.  A somewhat inverse perpsective is taken here in that it is assumed that $h$ is a homeomorphism and dissimalrity is based on the error with respect to Eq. \ref{conj}.  These errors (and their $\Phi$-space instantiations), which are referred to as \emph{conjugacy residuals} here,  serve as vector components in an objective function space. Deviations from conjugacy pseudometrics are then defined based on magnitudes of these vectors.  

The conjugacy residuals follow from rewriting Eq. \ref{conj} as
\begin{equation}
\label{res_conj}
\begin{split}
& y=h(x)+ \hat{r}_1(f,g,h), \\
& h \circ f(x) = g \circ h(x) + \hat{r}_2(f,g,h).
\end{split}
\end{equation}
where residual vectors $\hat{r}_1,\hat{r}_2 \in \mathbb{C}^{N_{\Psi}}$ are zero if the systems are topologically conjugate and non-zero otherwise.  Therefore we seek to find $h$ that minimizes $\hat{r}_1$ and $\hat{r}_2$. Two immediate observations are worth noting: first, in general the $h$ that minimizes $\hat{r}_1$ can only be guaranteed to minimize $\hat{r}_2$ when the two systems are conjugate (i.e.  $\hat{r}_1=\hat{r}_2=0$ for some $h$) and second the best combinations of $\hat{r}_1$ and $\hat{r}_2$ can only be quantified in a Pareto optimal sense since minimizing $\hat{r}_1$ will not in general minimize $\hat{r}_2$ away from conjugacy.  The transformation $h$ is restricted to be in the group of invertible matrices, i.e. $h = C \in GL(N_{\Psi})$, which gaurantees that only homeomorphisms are considered.   Therefore in the data-driven context of Eqs. \ref{sys_conj_eig} and \ref{traj_conj_eig}, the problem of finding the best residuals can be treated by minimizing the two scalar valued conjugacy residuals $r_1,r_2$:
\begin{equation}
\label{mopt_r}
\begin{split}
& \min_{C \in GL(N_{\Psi})}  \left(r_1,r_2 \right), \\
& r_1 =  \left\Vert \Phi_g - C \Phi_f \right\Vert_F, \;r_2 =  \left\Vert \Lambda_f - C^{-1} \Lambda_g C \right\Vert_F. \; 
\end{split}
\end{equation}
Appropriately normalized conjugacy residuals can also be used if desired. For instance, one could use 
\begin{equation}
\label{r1scaled}
r_1 =  \left\Vert \Phi_g - C \Phi_f \right\Vert_F/\left\Vert \Phi_{ref} \right\Vert_F
\end{equation}
and 
\begin{equation}
\label{r2scaled}
r_2 = \left\Vert \Lambda_f - C^{-1} \Lambda_g C \right\Vert_F/\left\Vert \Lambda_{ref} \right\Vert_F
\end{equation}
where $\Phi_{ref}$ and $\Lambda_{ref}$ correspond to a chosen reference system.

The conjugacy residuals can be viewed as vector components in objective function space, where the origin corresponds to topologically conjugacy.  Based on these residuals, deviation from conjugacy is defined as follows:
\begin{definition}
Deviation from conjugacy, $d(r_1,r_2)$, is defined as $d = \sqrt{r_1^2 + r_2^2}$ where $(r_1,r_2)$ correspond to Pareto points of Eq. \ref{mopt_r}.   
\end{definition}

\subsection{Simplifications to the Pareto optimality problem}

Up to this point, no restrictions on group membership of $C$ have been made other than it must be invertible.  However attempting to solve the non-convex multi-objective function optimization problem given by Eq.\ref{mopt_r} will not be scalable and likely intractable. For example, consider the case where $K_f,K_g \in \mathbb{R}^{N_{\Psi} \times N_{\Psi}}, N_{\Psi}=10$. In this case, $C$ would have 100 independent variables. Therefore, even for a fairly small $N_{\Psi}$, global optimization would be quite time consuming while local optimization would rely on having good initial guesses in order to obtain the Pareto optimal solutions to Eq. \ref{mopt_r}. Clearly solving Eq. \ref{mopt_r} is not suitable for a variety of machine learning and optimization analyses that would require solutions for many combinations of dynamical systems.  It is demonstrated below that the problem can be simplified considerably by considering the transformation requirements for recognizing topological conjugacy as well as Pareto front geometry in two-dimensional objective function space.

\subsubsection{Solutions over the group of unitary transformations}
\label{unitary_subsec}

The first simplification that is considered is based on the transformation attributes that are required to capture topological conjugacy between two systems. When considering $C \in GL(N_{\Psi})$, this implies transformations that will account for rotation, reflection, and stretching.  Thus, an important question is whether all three transformation components are necessary?  This is addressed in Theorem \ref{unitary_opt_soln} below.


\begin{theorem}
\label{unitary_opt_soln}
If two topologically conjugate systems have discrete spectra with Koopman eigenfunctions scaled such that $\max(\Phi_{f,i}),\max(\Phi_{g,i}) = 1 \forall i =1 \ldots N_{\Psi}$, then there is a unitary matrix $C \in U(N_{\Psi})$ such that $\left\Vert \Phi_g - C \Phi_f \right\Vert_F = 0$ and $\\ \left\Vert \Lambda_f - C^{-1} \Lambda_g C \right\Vert_F = 0$.
\end{theorem}

\begin{proof}
First, consider that the $ith$ row of $\Phi \in {N_{\Psi} \times N_{\Delta t}}$ is given by $\phi_{0,i} [\lambda_i  \cdots \lambda^{N_{\Delta t}}_i]$ where $\phi_{0,i}$ denotes the value of the eigenfunction corresponding to $\lambda_i$ at $t=0$.  So the rows of $\Phi$ are just simple harmonic signals in discrete time where $\phi_{0,i}$ dictates amplitude and phase while $\lambda_i$ dictates decay rate and frequency. With this perspective in mind the term inside the norm operator of Eq. \ref{traj_conj_eig} is
\begin{equation}
\label{signals}
\left[
\begin{matrix}
\phi_{0,i} \lambda_{i}  \cdots \phi_{0,i} \lambda^{N_{\Delta t}}_{i} \\
 \vdots  \\
\phi_{0,N_{\Psi}} \lambda_{N_{\Psi}}  \cdots \phi_{0,N_{\Psi}} \lambda^{N_{\Delta t}}_{N_{\Psi}}
\end{matrix}
\right]_g
=C\left [
\begin{matrix}
\phi_{0,i} \lambda_{i}  \cdots \phi_{0,i} \lambda^{N_{\Delta t}}_{i} \\
 \vdots  \\
\phi_{0,N_{\Psi}} \lambda_{N_{\Psi}}  \cdots \phi_{0,N_{\Psi}} \lambda^{N_{\Delta t}}_{N_{\Psi}}
\end{matrix}
\right]_f .
\end{equation}
Note that the rows of these matrices will not be aligned since it is assumed the eigensolutions for the two systems are arbitrarily ordered. Since $C$ must be unitary to prove the theorem,  $C$ can be set to the composition of matrix $\Theta \in U(N_{\Psi})$ and permutation matrix $P \in U(N_{\Psi})$, leading to $C=\Theta P$. Note the standard definition of a permutation matrix is one in which the sum of elements in each column is one,  the sum in each row is one, and the elements of the matrix can only be one or zero.  This results in
\begin{equation}
\label{perm_f}
\left[
\begin{matrix}
\phi_{0,i} \lambda_{i}  \cdots \phi_{0,i} \lambda^{N_{\Delta t}}_{i} \\
 \vdots  \\
\phi_{0,N_{\Psi}} \lambda_{N_{\Psi}}  \cdots \phi_{0,N_{\Psi}} \lambda^{N_{\Delta t}}_{N_{\Psi}}
\end{matrix}
\right]_g
=\Theta P \left [
\begin{matrix}
\phi_{0,i} \lambda_{i}  \cdots \phi_{0,i} \lambda^{N_{\Delta t}}_{i} \\
 \vdots  \\
\phi_{0,N_{\Psi}} \lambda_{N_{\Psi}}  \cdots \phi_{0,N_{\Psi}} \lambda^{N_{\Delta t}}_{N_{\Psi}}
\end{matrix}
\right]_f.
\end{equation}
From Eq. \ref{perm_f}, $P$ will permute the rows of the matrix corresponding to system $f$. Without loss of generality, $P$ can be chosen to be the permutation matrix that aligns the rows of the matrix for system $f$ such that the rows for both system matrices correspond to the same eigenvalues. This leads to
\begin{equation}
\label{perm_f_short}
D_{\phi,g} L_g = \Theta D_{\phi,f}  L_g, 
\end{equation}
where
\begin{equation}
\begin{split}
& D_{\phi,g} = \left [
\begin{matrix}
\phi_{0,i} & & \\
& \ddots & \\
& & \phi_{0,N}
\end{matrix}
\right]_g ,
D_{\phi,f} = \left [
\begin{matrix}
\phi_{0,i} & & \\
& \ddots & \\
& & \phi_{0,N}
\end{matrix}
\right]_f \\
& L_g = \left[ \begin{matrix}
\lambda_{i}  \cdots  \lambda^{N_{\Delta t}}_{i} \\
 \vdots  \\
\lambda_{N}  \cdots \lambda^{N_{\Delta t}}_{N}
\end{matrix}
\right].
\end{split}
\end{equation}
Therefore the action of $P$ is to align the discrete time signals from Eq. \ref{signals} such that the decay rates and frequencies in each row of the system $g$ matrix matches with those in the system $f$ matrix.  

The phases between the signals are not yet the same.  This is resolved by solving for $\Theta$ from Eq. \ref{perm_f_short} which gives
\begin{equation}
\label{theta_unitary}
\Theta = D_{\phi,g} D^{-1}_{\phi,f}.
\end{equation}
Since the eigenfunctions have been assumed to have amplitudes equal to one, $\left\vert \phi_{0,i} \right\vert = 1 \forall i$. As a result, $D_{\phi,g}$ and $D_{\phi,f}$ are unitary which implies $D_{\phi,g} D^{-1}_{\phi,f}$ is unitary by composition properties of groups. Therefore $\Theta$ is unitary and can be interpreted as aligning the phases of the time signals from Eq. \ref{signals}. Finally, since $\Theta \in U(N_{\Psi})$ and $P \in U(N_{\Psi})$, then $C \in U(N_{\Psi})$ and $\left\Vert \Phi_g - C \Phi_f \right\Vert_F = 0$ at conjugacy.

Next, consider the operator residual $\left\Vert \Lambda_f - C^{-1} \Lambda_g C \right\Vert_F = 0$. Assuming that $C$ is unitary, it must be shown that \\ $\left\Vert \Lambda_f - C^{*} \Lambda_g C \right\Vert_F = 0$ at conjugacy, where the $*$ superscript denotes complex conjugate transpose. Using the Frobenius norm trace identity gives
\begin{equation}
\label{trace}
\begin{split}
& \left\Vert \Lambda_f - C^{*} \Lambda_g C \right\Vert_F = \Tr \left( \Lambda_f - C^{*} \Lambda_g C \right)\left( \Lambda_f^{*} - C^{*} \Lambda_g^{*} C \right) \\
& = \Tr(\Lambda_f \Lambda_f^*) + \Tr(\Lambda_g \Lambda_g^*) - 2 \Tr(C^* \Lambda_g C \Lambda_f^* ),
\end{split}
\end{equation}
where the cyclic prperty of the trace $\Tr$ operation and the trace product idenity $\Tr (AB^*) = \Tr(BA^*)$ were used to obtain the second line of Eq. \ref{trace}. Proceeding in a similar manner as before, it is assumed without loss of generality that $C = \Gamma P$ where $\Gamma,P \in U(N_{\Psi})$ and $P$ is the same permutation matrix in Eq. \ref{perm_f} that reorders the eigenvaues of system $f$ to match the ordering of eigenvalues in system $g$, i.e. $P \Lambda_f P^{\mathrm{T}} = \Lambda_g$.  Substituting $C = \Gamma P$ into the final term of Eq. \ref{trace} leads to
\begin{equation}
\label{trace2}
\begin{split}
& \Tr(\Lambda_f \Lambda_f^*) + \Tr(\Lambda_g \Lambda_g^*) - 2 \Tr(P^{\mathrm{T}} \Gamma \Lambda_g \Gamma P \Lambda_f^*)\\
&=\Tr(\Lambda_f \Lambda_f^*) + \Tr(\Lambda_g \Lambda_g^*) - 2 \Tr(\Gamma^* \Lambda_g \Gamma \Lambda_g^*) \\
&= \sum_i \left\vert \lambda_{f,i} \right\vert^2 +\left\vert \lambda_{g,i} \right\vert^2 -2 \lambda_{g_i} \bar{\lambda}_{g,i} \Gamma_{ii} \bar{\Gamma}_{ii} \\
& =  2 \sum_i \left\vert \lambda_{g,i} \right\vert^2 -2 \lambda_{g_i} \bar{\lambda}_{g,i} \Gamma_{ii} \bar{\Gamma}_{ii}
\end{split}
\end{equation}
where $\bar{\lambda}_{g,i}$ is the complex conjugate of the $i$th eigenvalue of system $g$ and $\bar{\Gamma}_{ii}$ is the complex conjugate of element $(i,i)$ of $\Gamma$.  The cyclic property of $\Tr$ and that fact that $P \Lambda_f^* P^{\mathrm{T}}=\Lambda_g^*$ were used to obtain the second line of Eq. \ref{trace2}, and Eq. \ref{eig_conj} was used to obtain the final line. From the last line of Eq. \ref{trace2} it is clear that topological conjugacy occurs when $\Gamma_{ii} \bar{\Gamma}_{ii} = 1$.  Since $\Gamma$ is unitary, this condition implies $\Gamma$ is any diagonal matrix where each diagonal element has magnitude equal to one. Therefore, the proof of Theorem \ref{unitary_opt_soln} is completed by choosing $\Gamma = \Theta$ such that the same unitary matrix $C$ satisfies both $\left\Vert \Lambda_f - C^{*} \Lambda_g C \right\Vert_F = 0$ and $\left\Vert \Phi_g - C \Phi_f \right\Vert_F = 0$.

\end{proof}

\begin{remark}
Assuming that the amplitudes of the Koopman eigenfunctions are equal to one is not restrictive since Koopman eigenfunctions are not unique \cite{koop_dictionary}; i.e. if $\phi$ is an eigenfunction corresponding to eigenvalue $\lambda$ then $c \phi$ is also an eigenfunction corresponding to $\lambda$ for any constant $c \in \mathbb{C}$. 
\end{remark}





\begin{remark}
Theorem \ref{unitary_opt_soln} shows that it is sufficient to only consider transformations involving rotation and reflection (in $\Phi$-space) to identify that two systems are equivalent. Therefore, it is reasonable to simplify Eq.\ref{mopt_r} to
\begin{equation}
\label{mopt_r_unitary}
\begin{split}
& \min_{C \in U(N_{\Psi})}  \left(r_1,r_2 \right), \\
& r_1 =  \left\Vert \Phi_g - C \Phi_f \right\Vert_F, \;r_2 =  \left\Vert \Lambda_f - C^{*} \Lambda_g C \right\Vert_F.
\end{split}
\end{equation}
Theorem \ref{unitary_opt_soln} implies that Eq. \ref{mopt_r_unitary} can be expected to be quantitatively accurate when two systems are close to conjugate since stretching is an unnecessary transformation component at conjugacy.  Indeed, stretching would be expected to be play a more significant role in precise quantification of deviation when two systems are very far from topological conjugacy. However, in such cases, Eq. \ref{mopt_r_unitary} would register large values of $r_1$ and/or $r_2$ thus capturing the relative argument that the two systems are very different from the perspective of topological conjugacy. Thus Eq. \ref{mopt_r_unitary} is expected to be qualitatively accurate very far from conjugacy.
\end{remark}

\begin{remark}
\label{mezic_metric}
It is interesting to consider the implication from Theorem \ref{unitary_opt_soln} that $\Gamma$ must be unitary and diagonal in order to maintain consistency with topological conjugacy. This imples that \emph{any} unitary diagonal $\Gamma$ will minimize $r_2$ since \\$(\Gamma P)^*\Lambda_g (\Gamma P) = P^*\Lambda_g P$.  From a purely mathematical perspective, it is known that the optimal solution for minimizing $r_2$ in isolation is a permutation matrix \cite{hoffman}. Similalrly, taking a spectral operator perspective leads to the same solution as a result of minimizing the Wasserstein distance operating on the Koopman eigenvalues \cite{redman2022,redman2024}.  So the  the main implication from Theorem \ref{unitary_opt_soln} that optimal solutions can be found by searching over unitary transformations is in agreement with the solutions from \cite{hoffman} and \cite{redman2022,redman2024} from the perspective of operator level dissimilarity measures. However, Theorem \ref{unitary_opt_soln} also suggests a more general solution than \cite{hoffman} and \cite{redman2022,redman2024} since $\Gamma$ is effectively a free paramater that can be used to minimize $r_1$.  Thus, Theorem \ref{unitary_opt_soln} offers a unified perspective for searching for transformations that minimize \emph{both} $r_1$ and $r_2$ simultaneously. This will be exploited for deriving Pareto optimal solutions in Section \ref{solns}.

\end{remark}

\begin{remark}
\label{phi_advantage}
When solving Eq. \ref{sys_conj} and Eq. \ref{traj_conj} in $\Psi$-space by restricting $T$ to be unitary(or orthogonal), e.g.  \cite{6247929,lds_distance_opt_algorithm,williams_shape,ostrow_DSA}, the resulting pseudometrics will not yield zero dissimilarity for topologically conjugate systems in general.  Since conjugacy residuals satisfy Theorem \ref{unitary_opt_soln}, the pseudometrics in this study will recover zero when $f$ and $g$ are topologically conjugate. Therefore the computational efficiency associated with the approaches in \cite{6247929,lds_distance_opt_algorithm,williams_shape,ostrow_DSA} are achieved while maintaining theoretical consistency with topological conjugacy by computing residuals in $\Phi$-space.
\end{remark}

\subsubsection{Geometry of Pareto optimality}

Even after simplifying by restricting the search for optimal solutions to the group of unitary transformations, numerical optimization methods would still be required to solve the non-convex Eq. \ref{mopt_r_unitary} for all Pareto points. This approach would not be solvable in polynomial time, and thus not scalable.  Therefore a key consideration at this point is how many Pareto optimal solutions of Eq. \ref{mopt_r_unitary} are necessary to develop pseudometrics representative of the entire Pareto front in an approximate manner? Theorem \ref{avg_pareto} below shows that the answer is only two, as along as they are the Pareto points corresponding to independent minimization of $r_1$ and $r_2$.  

\begin{theorem}
\label{avg_pareto}
Given two systems $f$ and $g$ with the optimal solution that minimizes $r_1$ given by $C_{r_1} \in U(N_{\Psi})$, and the optimal solution that minimizes $r_2$ given by $C_{r_2} \in U(N_{\Psi})$,  then over all possible Pareto optimal solutions to Eq. \ref{mopt_r_unitary}, the minimum, maximum, and average deviations from conjugacy are $d_{min}(f,g)$, $ d_{max}(f,g)$, and $d_{avg}(f,g)$ respectively where 
\begin{equation}
\label{dmin}
d_{min}(f,g) \approx \sqrt {r^2_1(C_{r_1})+r^2_2(C_{r_2})},
\end{equation}
\begin{equation}
\label{dmax}
d_{max}(f,g) \approx \sqrt {r^2_1(C_{r_2})+r^2_2(C_{r_1})},
\end{equation}
and
\begin{equation}
\label{davg}
\begin{split}
 & d_{avg}(f,g) \approx [r_1(C_{r_2}) r_2(C_{r_1}) A(2r_1(C_{r_2}),2r_2(C_{r_1})) - \\
 & r_2(C_{r_2}) r_1(C_{r_2}) A(2r_2(C_{r_2}),2r_1(C_{r_2})) - r_2(C_{r_1}) r_1(C_{r_1}) A(2r_2(C_{r_1}),2r_1(C_{r_1})) \\
 &+ r_1(C_{r_1}) r_2(C_{r_2}) A(2r_1(C_{r_1}),2r_2(C_{r_2}))] \\
 & / [r_1(C_{r_2}) r_2(C_{r_1}) - r_2(C_{r_2}) r_1(C_{r_2})-\\
 & r_2(C_{r_1}) r_1(C_{r_1})+r_1(C_{r_1}) r_2(C_{r_2})]
\end{split}
\end{equation}
where
\begin{equation}
\begin{split}
& A(\Delta_1,\Delta_2) = \Delta_2^3 \sinh^{-1}(\Delta_1/ \Delta_2)+\Delta_1^3 \sinh^{-1}(\Delta_2/ \Delta_1) \\
& + 2 \Delta_1 \Delta_2 \sqrt{\Delta_1^2+\Delta_2^2} ]/(12\Delta_1 \Delta_2).
\end{split}
\end{equation}

\end{theorem}

\begin{proof}
Theorem \ref{avg_pareto} is proved by considering the geometry of a two-dimensional Pareto front, as depicted in Fig. \ref{pareto}. The known Pareto optimal solutions are depicted in Fig. \ref{pareto} as filled circles at the top left and bottom right corners of the bounded region. The top left point is located at $(r_1(C_{r_1}), r_2(C_{r_1}))$ and the bottom right point is at \\$(r_1(C_{r_2}), r_2(C_{r_2}))$.  

Although the remaining Pareto points are unknown, they can be bounded. From the bottom right Pareto point, consider an infinitesimally small parameter $\epsilon \ll 1$. Since the bottom right point is the best solution possible for $r_2$, all points with $r_1=(1-\epsilon)r_1(C_{r_2})$ and $r_2(C_{r_2})< r_2 \le (1-\epsilon)r_2(C_{r_1})$ are Pareto optimal with respect to the two known Pareto optimal points. This vertical line forms the right edge of the bounded region in Fig.\ref{pareto}, and the point at the top of the line corresponds to the upper right corner of the bounded region. The top edge is formed by the line through the top right corner and the known Pareto point at the top-left. The left and bottom edges of the bounded region are similarly formed by noting that the bottom left corner of the region occurs at $((1+\epsilon)r_1(C_{r_1}), (1+\epsilon)r_2(C_{r_2}))$. The vector magnitudes from the origin to the bottom left corner, $d_{min}$, and from the origin to the top right corner $d_{max}$ immediately follow from the geometry of the bounded region,
\begin{equation}
\label{dmin_approx}
\begin{split}
& d_{min} = (1+\epsilon)\sqrt {r^2_1(C_{r_1})+r^2_2(C_{r_2})} \\
& \approx \sqrt {r^2_1(C_{r_1})+r^2_2(C_{r_2})},
\end{split}
\end{equation}
and
\begin{equation}
\label{dmax_approx}
\begin{split}
& d_{max}=(1-\epsilon) \sqrt {r^2_1(C_{r_2})+r^2_2(C_{r_1})} \\
&  \approx \sqrt {r^2_1(C_{r_2})+r^2_2(C_{r_1})}.
\end{split}
\end{equation}
The average deviation from conjugacy over all points in the rectangle depicted in Fig. \ref{pareto} is given by Eq. \eqref{davg} which follows directly from \cite{rectangle_dist}.

\begin{figure}
  \includegraphics[width=0.45\textwidth]{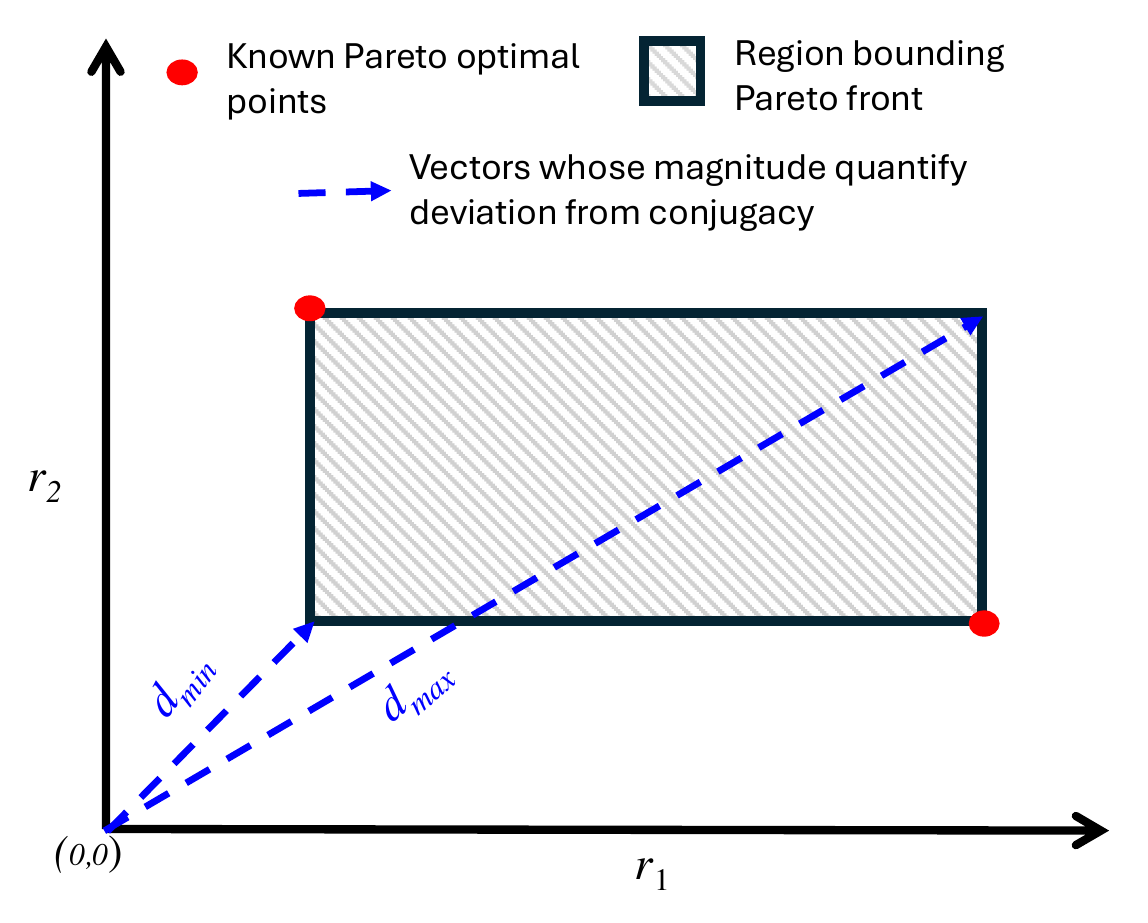}
\caption{Geometric bounds on all possible Pareto optimal solutions; the origin corresponds to topological conjugacy}
\label{pareto}       
\end{figure}

\end{proof}

\begin{remark}
Theorem \ref{avg_pareto} simplifies the problem considerably since analytical solutions can be found for the two independent minimization problems (derived next in Section \ref{solns}). As a result, numerically solving Eq. \ref{mopt_r_unitary} for all Pareto optimal points is not necessary since representative values of $d$ that account for all possible Pareto optimal solutions can be efficiently computed in an approximate manner. 
\end{remark}

\begin{remark}
All deviations from conjugacy in Eqs. \ref{dmin}-\ref{davg} are pseudometrics; i.e.  $d(a,b)=0$ when $a \sim b$, $d(a,b)=d(b,a)$, and $d(a,b) \le d(a,c) + d(b,c)$ \cite{metrics}. This follows from the fact that Eq. \ref{mopt_r_unitary} is solved over the group of unitary matrices. Since unitary matrices are isometric, $r_1$ and $r_2$ will be pseudometrics \cite{williams_shape,ostrow_DSA}. Furthermore, since $d_{min}$, $d_{max}$, and $d_{avg}$, are derived as Euclidean distances of vectors with components $r_1$ and $r_2$,  they are also pseudometrics. Pseudometrics, as opposed to semimetrics that do not obey the triangle inequality, provide favorable properties for use in machine learning applications \cite{williams_shape}. 
\end{remark}

\subsection{Analytical solutions in $\Phi$-space}
\label{solns}
Theorems \ref{unitary_opt_soln} and \ref{avg_pareto} together lead to two substantial simplifications: the first is that only solutions in the unitary group need to be found as opposed to the general linear group, and the second is that within the unitary group only two solutions need to be found --  $C_{r_1}$ and $C_{r_2}$.  These analytical solutions and their computational complexity are described next.

First, consider the solution for $C_{r_1}$. The optimal transformation matrix $C_{r_1}$ is obtained by minimizing $r_1$,
\begin{equation}
\label{phi_1procrustes}
\min_{C \in U(N_{\Psi})} \left\Vert \Phi_g - C \Phi_f \right\Vert_F.
\end{equation}
The optimal solution to Eq. \ref{phi_1procrustes} is given by Sch\"{o}nemann's solution to the orthogonal Procrustes problem \cite{schonemann},
\begin{equation}
\label{Cr1}
C_{r_1} = U_{\phi} V_{\phi}^*, \; U_{\phi}, V_{\phi} \in U(N_{\Psi}),
\end{equation}
where the columns of $U_{\phi}$ and $V_{\phi}$ contain the left and right singular vectors of the matrix $\Phi_g \Phi_f^*$. Widely available singular value decomposition solvers, such as the one available in MATLAB, can solve for singular vectors in $\mathcal{O}(N_{\Psi}^3)$ time making Eq. \ref{phi_1procrustes} computationally efficient to solve.

The solution for $C_{r_2}$ is obtained by solving 
\begin{equation}
\label{phi_2procrustes}
\min_{C \in U(N_{\Psi})} \left\Vert \Lambda_f - C^{*} \Lambda_g C \right\Vert_F.
\end{equation}
Since $\Lambda_f$ and $\Lambda_g$ are normal, it is known from \cite{hoffman} that the optimal solution for $C_{r_2}$ must act to permute the eigenvalues stored in $\Lambda_g$ to optimally align with the ordering in $\Lambda_f$ such that $\sum_i \vert \lambda_{f,i} - \lambda_{g,i} \vert^2$ is minimized.  The solution for the optimal permutation matrix $P$ is obtained by solving the linear assignment problem with elements of the distance matrix given by $D_{ij} = \vert \lambda_{f,i} -\lambda_{g,j} \vert^2$. Standard algorithms can solve the linear assignment problem in  $\mathcal{O}(N_{\Psi}^3)$ time. 

While an optimal permutation matrix $P$ is \emph{one} solution to Eq. \ref{phi_2procrustes}, in general it is not the solution to the conjugacy residual problem given by Eq. \ref{mopt_r_unitary}; see Remark \ref{mezic_metric}. Since any solution $C_{r_2} = \Gamma P$ such that $\sum_i \vert \lambda_{f,i} - \lambda_{g,i} \vert^2$ is minimized will minimize Eq. \ref{phi_2procrustes}, one also requires a solution for $\Gamma$. Thus a more general solution than a pure permutation matrix is required for computing deviation from conjugacy.  The solution for  $\Gamma$ proposed here is obtained by finding the best solution (in a least squares sense) that minimizes Eq. \ref{phi_1procrustes}.  Substituting $C_{r_2}=\Gamma P$ into Eq. \ref{phi_1procrustes} gives
\begin{equation}
\min_{\Gamma \in U(N_{\Psi})} \left\Vert \Phi_g - \Gamma P \Phi_f \right\Vert_F.
\end{equation}
If one were not concerned with constraints on $\Gamma$ the least squares solution $\tilde{\Gamma}=\Phi_g(P \Phi_f)^{\dagger}$, where $\dagger$ denotes the Moore-Penrose pseudoinverse, would suffice. Noting, that $\Gamma$ must be diagonal restricts the solution to $\Delta_{\tilde{\Gamma}}=Diag(\tilde{\Gamma})$ where $Diag$ denotes the diagonal matrix formed from the diagonal elements of $\tilde{\Gamma}$. The restriction that $\Gamma$ is unitary leads to
\begin{equation}
\label{gamma_opt}
\Gamma = U_{\Delta_{\tilde{\Gamma}}} V_{\Delta_{\tilde{\Gamma}}}^*,
\end{equation}
where the columns of $U_{\Delta_{\tilde{\Gamma}}}$ and $V_{\Delta_{\tilde{\Gamma}}}$ contain the singular vectors of $\Delta_{\tilde{\Gamma}}$. Therefore the solution for $C_{r_2}$ is
\begin{equation}
\label{Cr2}
C_{r_2} = U_{\Delta_{\tilde{\Gamma}}} V_{\Delta_{\tilde{\Gamma}}}^* P,
\end{equation}
which can be computed in $\mathcal{O}(N_{\Psi}^3)$ time.

It is worth noting that $r_2({C_{r_2}})$ corresponds to the Wasserstein distance proposed in \cite{redman2022,redman2024} for any choice of $\Gamma$.  The solution for $\Gamma$ determines whether a Pareto optimal solution to Eq. \ref{mopt_r_unitary} is found.  For instance, assuming $\Gamma$ equals the identity matrix leads to $r_2(P)$, which is equivalent to the solution for $r_2(C_{r_2})$ when $\Gamma$ is found from Eq. \ref{gamma_opt}. However,  the solution $r_1(P)$ will be dominated (in a Pareto optimality sense) by the solution $r_1(C_{r_2})$ since equation Eq. \ref{Cr2} includes the possibility that $\Gamma$ is the identity matrix,but is more general.

\subsection{Solutions in $\Psi$-space}

While the primary focus here is computing deviation from conjugacy pseudometrics, one may be interested in analyzing systems in $\Psi$-space. Therefore, for the sake of completeness transformations in $\Psi$-space are recovered by substituting Eqs. \ref{Cr1} and \ref{Cr2} into Eq. \ref{TtoC}. Since solutions for $T$ are based on left-eigenvectors which are not unique, an optimal magnitude for the eigenvectors must be found.  Recognizing that $W_g$ can be arbitrarily scaled by a diagonal matrix $\Omega$, Eq. \ref{TtoC} becomes
\begin{equation}
\label{TfromC}
T_C = (\Omega W_g)^{-1}C_{r_{1,2}} W_f.
\end{equation}
Similar to the strategy used when solving for $\Gamma$, $\Omega$ is obtained by finding the matrix that makes $T_C$ closest to the least squares solution of Eq. \ref{traj_conj},$T_{LSQ}=\Psi_g \Psi_f^{\dagger}$, i.e.
\begin{equation}
\min_{\Omega} \left\Vert(\Omega W_g)^{-1}C_{r_{1,2}} W_f - T_{LSQ} \right\Vert_F.
\end{equation}
The solution for $\Omega^{-1}$ is then given by
\begin{equation}
\Omega^{-1} = Diag( W_g T_{LSQ}(C_{r_{1,2}} W_f)^{\dagger} ).
\end{equation}

The advantages of searching for an optimal transformation in $\Phi$-space become apparent when one recognizes that although $C_{r_{1,2}}$ are unitary, $T_C$ is not in general.  Since the pseudoinverse is $\mathcal{O}(N_{\Psi}^3)$ computational complexity, Eq. \ref{TfromC} can be computed in polynomial time. In contrast, directly solving for $T$ without an analytical solution requires solving the two-sided Procrustes probem \cite{MENG2022108334} given by
\begin{equation}
\label{2sidedProc}
\min_{T\in GL(N_\Psi)} \left\Vert K_f- T^{-1} K_g T\right\Vert_F.
\end{equation}
Since there are no polynomial-time solutions to Eq. \ref{2sidedProc} \cite{MENG2022108334}, numerical optimization algorithms for non-convex problems would be required. Searching over the general linear group implies an optimization problem in $N_{\Psi}^2$ variables.  Restricting the problem to special orthogonal transformations requires solving an optimization problem in $N_{\Psi}(N_{\Psi}-1)/2$ variables \cite{ostrow_DSA,lds_distance_opt_algorithm}.  For context, an algorithm such as the one presented in \cite{lds_distance_opt_algorithm} requires $\mathcal{O}(N_{\Psi}^3)$ time for a \emph{single} iteration and numerous iterations would be required for convergence. Coupled with the discussion in Remark \ref{phi_advantage}, it is clear that solving for optimal transformations in $\Phi$-space offers substantial computational efficiency, scalability, and theoretical consistency advantages.  


Finally, it is interesting to note how viewing $T_C$ as an initial guess for $T$ compares to other approaches to generating initial guesses for $T$. These initial guesses are used in heuristics for solving the two-sided Procrustes problem \cite{MENG2022108334}.  Other approaches reported in the literature require that $K_{f,g}$ be symmetric or that $K_{f,g}$ be transformed into Hermitian matrices based on their symmetric and skew-symmetric components in order to generate an initial guess for $T$ \cite{umeyama,MENG2022108334}.  Both assumptions have limited suitability when dealing with dynamical systems.  In contrast, the approach described here produces an initial guess that is directly interpretable from the perspective of dynamical systems without the need for ad-hoc transformations of $K_{f,g}$. 

\section{Results}
\label{results}

Results are presented for a simple benchmarking problem followed by an engineering example. The benchmarking problem was created by utilizing two paramaters that control deviation from conjugacy in an easily interpretable manner.  As an illustration of practical application, the engineering example illustrates how the deviation from conjugacy method can be used to compare biological mechanical actuation and man-made or bio-inspired mechanical actuation from the perspective of how information is dynamically exchanged between the morphology and control system. All computations were implemented in MATLAB.

\subsection{Benchmarking example}

Consider the systems $f$ and $g$ given by Eqs. \ref{sys_f} and \ref{sys_g} respectively.
\begin{equation}
\label{sys_f}
\frac{d}{dt}
\left[ \begin{matrix}
x_1 \\
x_2
\end{matrix}
\right]
= 
\left[ \begin{matrix}
\mu x_1 \\
\lambda (x_2-x_1^2)
\end{matrix}
\right]
\end{equation}

\begin{equation}
\label{sys_g}
\frac{d}{dt}
\left[ \begin{matrix}
y_1 \\
y_2
\end{matrix}
\right]
= 
\left[ \begin{matrix}
(2\alpha\mu-\beta\lambda) y_1 + 2(\alpha\mu-\beta\lambda) y_2 +\beta\lambda(y_1+y_2)^2 \\
(\beta\lambda-\alpha\mu) y_1 + (2\beta\lambda-\alpha\mu) y_2 -\beta\lambda(y_1+y_2)^2
\end{matrix}
\right].
\end{equation}
When $\alpha=1$ and $\beta=1$ then systems $f$ and $g$ are topologically conjugate where the homemorphism $h$ is given by 
\begin{equation}
\label{h_xy}
h = 
\left[
\begin{matrix}
2 & -1 \\
-1 & 1
\end{matrix}
\right].
\end{equation}
Deviation from conjugacy will increase from zero when $\alpha,\beta \neq 1$. 

Conveniently, systems $f$ and $g$ also have analytical Koopman operators. From \cite{brunton_invariant}, with Koopman observable vector given by $\Psi_f = \left[ \begin{matrix}x_1 & x_2 & x_1^2\end{matrix}\right]^T$, the Koopman operator for $f$ is
\begin{equation}
\label{K_f_example}
K_f = 
\left[
\begin{matrix}
\mu & 0 & 0 \\
0 & \lambda & -\lambda \\
0 & 0 & 2\mu
\end{matrix}
\right].
\end{equation}
Proceeding in a similar manner as in \cite{brunton_invariant} with \\$\Psi_g = \left[ \begin{matrix}y_1 & y_2 & (y_1+y_2)^2\end{matrix}\right]^T$, the Koopman operator for system $g$ is
\begin{equation}
\label{K_g_example}
K_g = 
\left[
\begin{matrix}
2\alpha\mu-\beta\lambda & 2\alpha\mu-2\beta\lambda & \beta\lambda \\
\beta\lambda-\alpha\mu &2\beta\lambda-\alpha\mu & -\beta\lambda \\
0 & 0 & 2\alpha\mu
\end{matrix}
\right].
\end{equation}
Data is generated for $\lambda = -0.1 + 10 i$ and $\mu=-0.001 + i$ and a uniform grid over the parameters $\alpha$ and $\beta$ by varying them from 0.1 to 2 in increments of 0.01.  Discrete time data is generated for $N_{\Delta t}=1000$ time steps using a discretization of $\Delta t = 0.01$.  At conjugacy, the non-orthogonal transformation in $\Psi$-space between the two systems is
\begin{equation}
\label{Texample}
T = 
\left[
\begin{matrix}
-1 & -(0.813+0.002i) & 0.703+0.004i \\
0.5 & 0.813+0.002i  & -(0.703+0.004i) \\
0 & 0 & 0.25
\end{matrix}
\right].
\end{equation}

\begin{figure*}
\subfigure(a) { \includegraphics[width=0.3\textwidth]{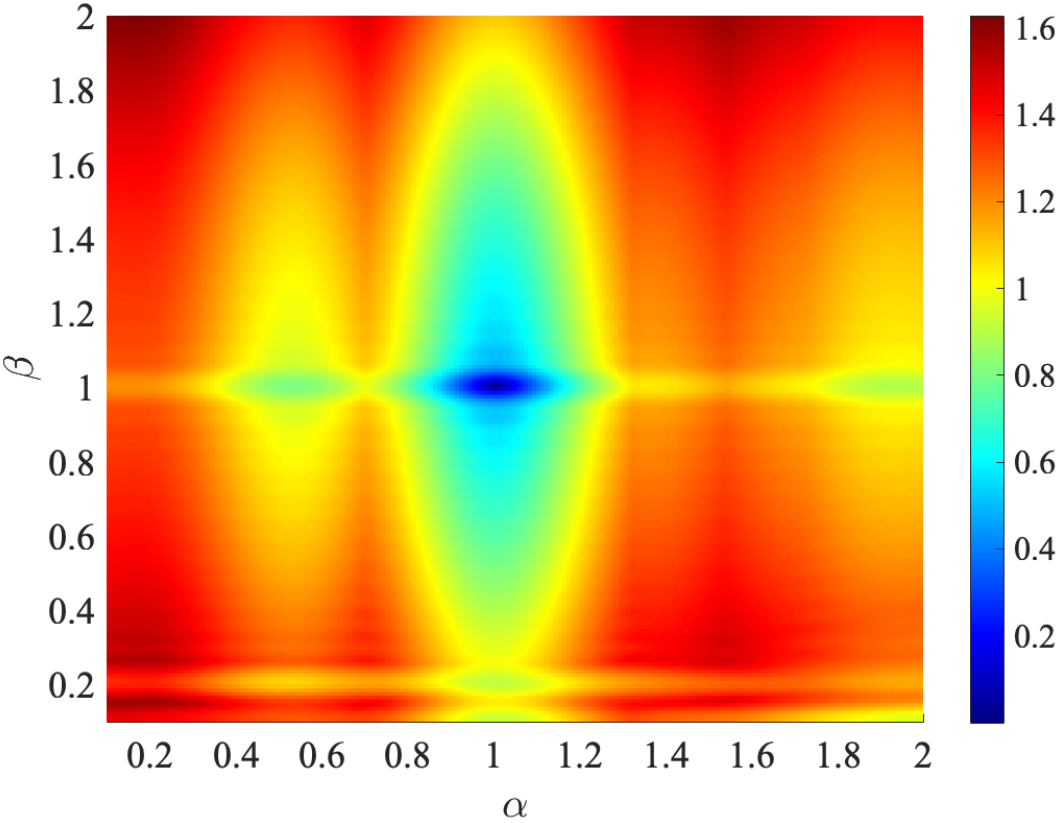}}
\subfigure(b) { \includegraphics[width=0.3\textwidth]{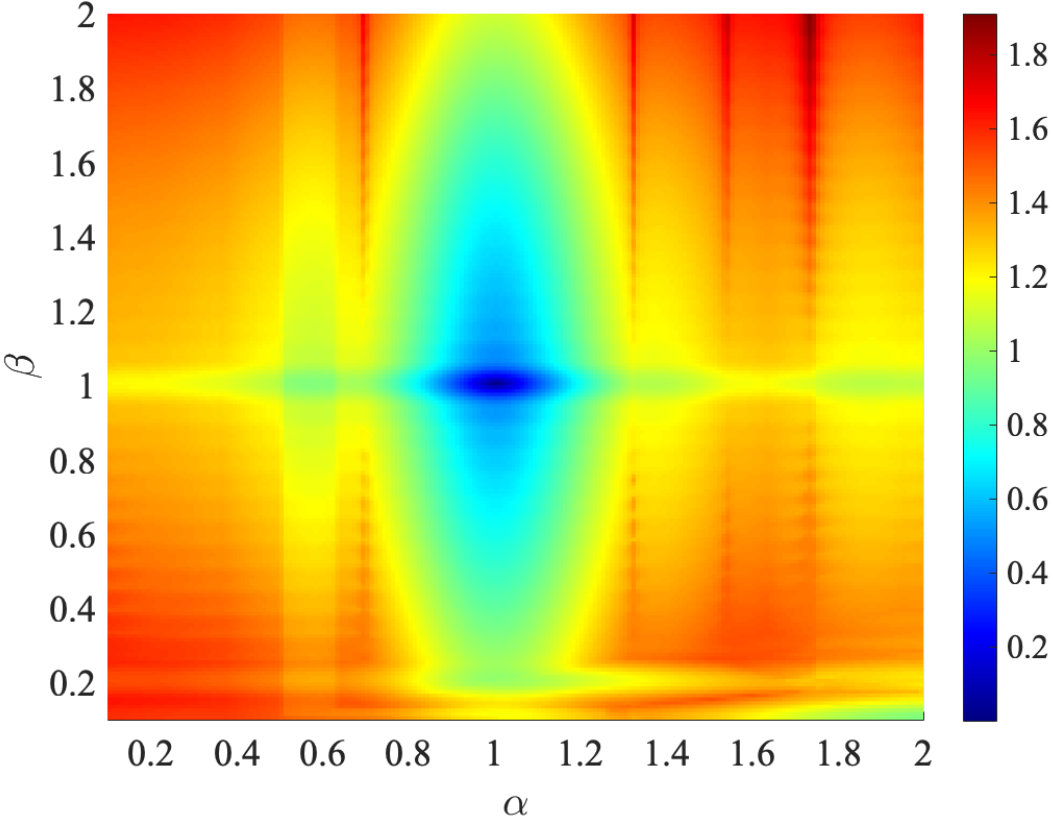}}
\subfigure(c) { \includegraphics[width=0.3\textwidth]{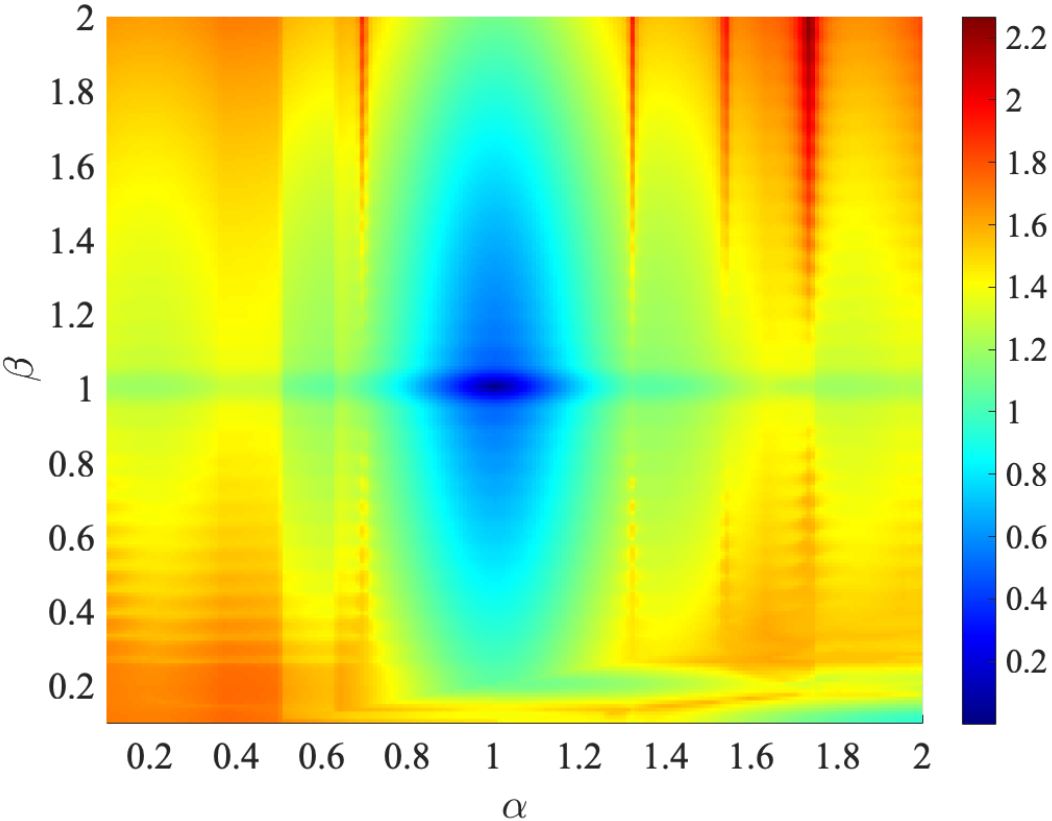}}
\caption{Deviations from conjugacy varied over $\alpha$ and $\beta$; (a) $d_{min}$ (b) $d_{avg}$ (c) $d_{max}$}
\label{dplots}       
\end{figure*}
Deviation from conjugacies corresponding to normalized conjugacy residuals from Eqs. \ref{r1scaled} and \ref{r2scaled} with system $f$ serving as the reference system are plotted in Fig. \ref{dplots}.  All three deviaton from conjugacy pseudometrics report $d = 7e^{-13}$ demonstrating that topological conjugacy is recovered at $\alpha,\beta=1$. Computing $T_C$ via Eq. \ref{TfromC} using the unitary transformations $C_{r_{1,2}}$ recovers the non-unitary transformation matrix given by Eq. \ref{Texample}. This demonstrates that the pseudometrics based on Theorem \ref{unitary_opt_soln} and unitary transformation in $\Phi$-space are able to capture conjugacy between systems even if the original transformation in $\Psi$-space is not unitary.  This is illustrated in Fig. \ref{X1X0_conj} where trajectories based on $T_C$ corresponding to $C_{r_1}$ are compared to those transformed by $T_{LSQ}$ in Fig. \ref{X1X0_conj}. As expected,  Figs. \ref{X1X0_conj}(a) and (b) show the transformed trajectory $T\Psi_f$ exactly matches $\Psi_g$ when the systems are topologically conjugate.  Furthremore, $T_C=T_{LSQ}$ at topological conjugacy demonstrating that while stretching is not required in the $\Phi$-space transformations, it is recovered in the $\Psi$-space transformations. 
\begin{figure*}
\subfigure(a) { \includegraphics[width=0.45\textwidth]{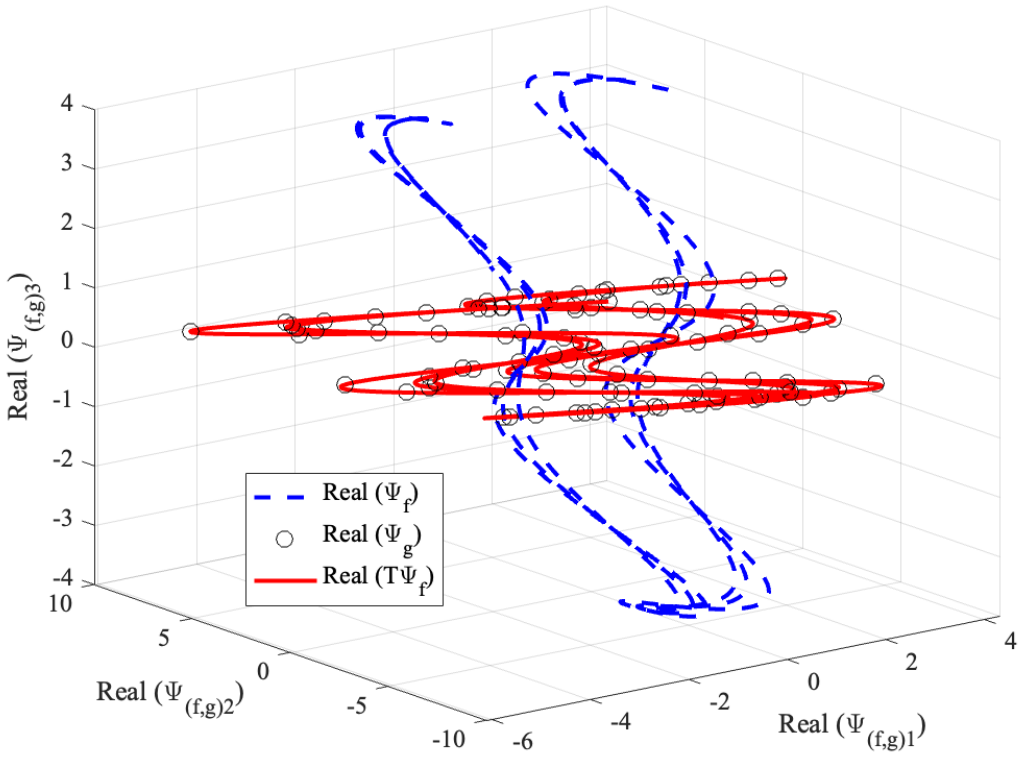}}
\subfigure(b) { \includegraphics[width=0.45\textwidth]{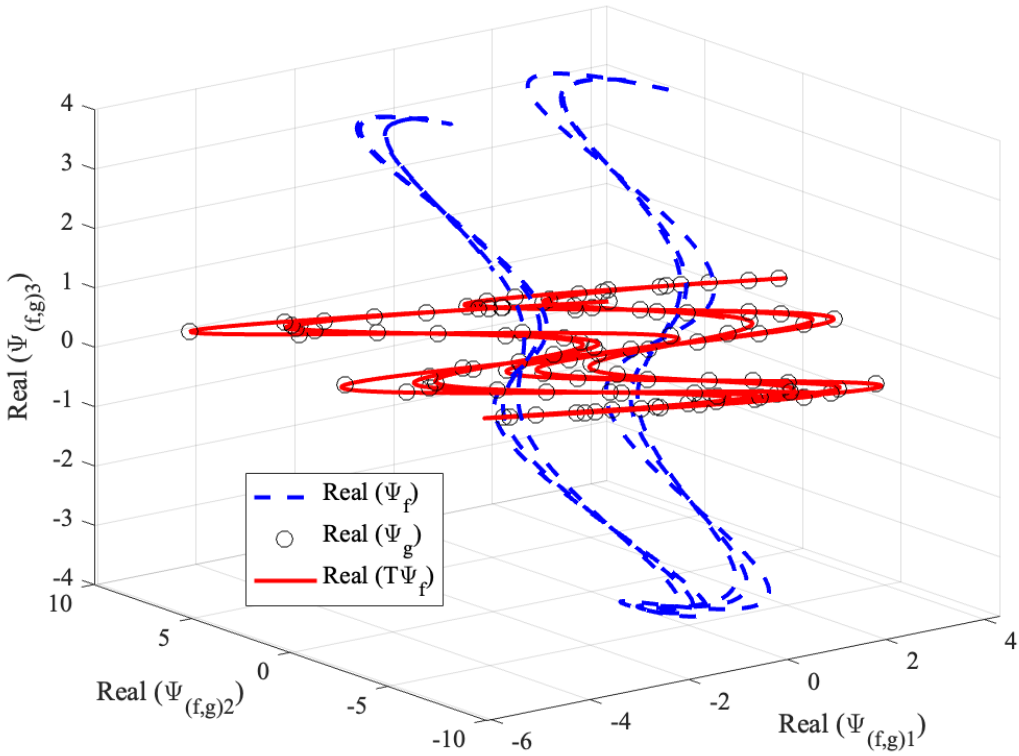}}
\caption{Comparisons of $\Psi_f$, $\Psi_g$, and transformed trajectory $T\Psi_f$ for $\alpha,\beta=1$; (a) $T_{C_{r_1}}$, (b)$T_{LSQ}$}
\label{X1X0_conj}       
\end{figure*}

Figure \ref{dplots} illustrates that all three pseudometrics show similar trends for deviation from conjugacy as one moves away from $\alpha,\beta=1$. The main difference is that $d_{min}$ has smaller regions of moderate deviations from conjugacy, while $d_{max}$ has larger such regions and $d_{avg}$ falls in between. This indicates that selection of which pseudometric to use should be made according to the preference of the user. For instance, if the user wishes to cluster/classify designs in a more binary fashion, i.e. systems are either far or close to each other then the user should select $d_{min}$. On the other hand if the user wishes a more conservative criterion that would lead to larger groupings of systems based on their deviations from conjugacy, then $d_{max}$ should be chosen, while $d_{avg}$ would fall in between.


Figure \ref{r1_r2_alpha_beta} shows conjugacy residuals when optimizing for each individually and demonstrates the importance of treating deviation from conjugacy as a Pareto optimality problem.  It is clear from Figs. \ref{r1_r2_alpha_beta}(a), and \ref{r1_r2_alpha_beta}(b) that comparing systems based only on $r_1$ or only on $r_2$ would lead to qualitatively different conclusions on dissimilarity between systems. For instance,  from Fig. \ref{r1_r2_alpha_beta}(b) the operator residual $r_2(C_{r_2})$ alone may lead one to conclude that all systems $g$ with $\beta$ close to one are relatively similar to system $f$. However, Fig. \ref{r1_r2_alpha_beta}(a) which accounts for differences in trajectory geometry indicates that this may not be the case depending on $\alpha$.  Since $\alpha$ multiplies $\mu$, it affects two of the three eigenvalues of $K_g$. Thus it is expected that it should impact whether system $g$ is close or far from $f$. Since the deviations from conjugacy account for both residuals, this expected result is captured in Fig. \ref{dplots}.  Similar conclusions can be made about comparisons based only on $r_1$ without accounting for $r_2$. Furthermore, Fig. \ref{r1_r2_alpha_beta}(c) demonstrates that the single objective optimal unitary transformation solutions are not the same in general and thus optimizing for one residual may lead to considerable increases in the other.
\begin{figure*}
\subfigure(a) { \includegraphics[width=0.3\textwidth]{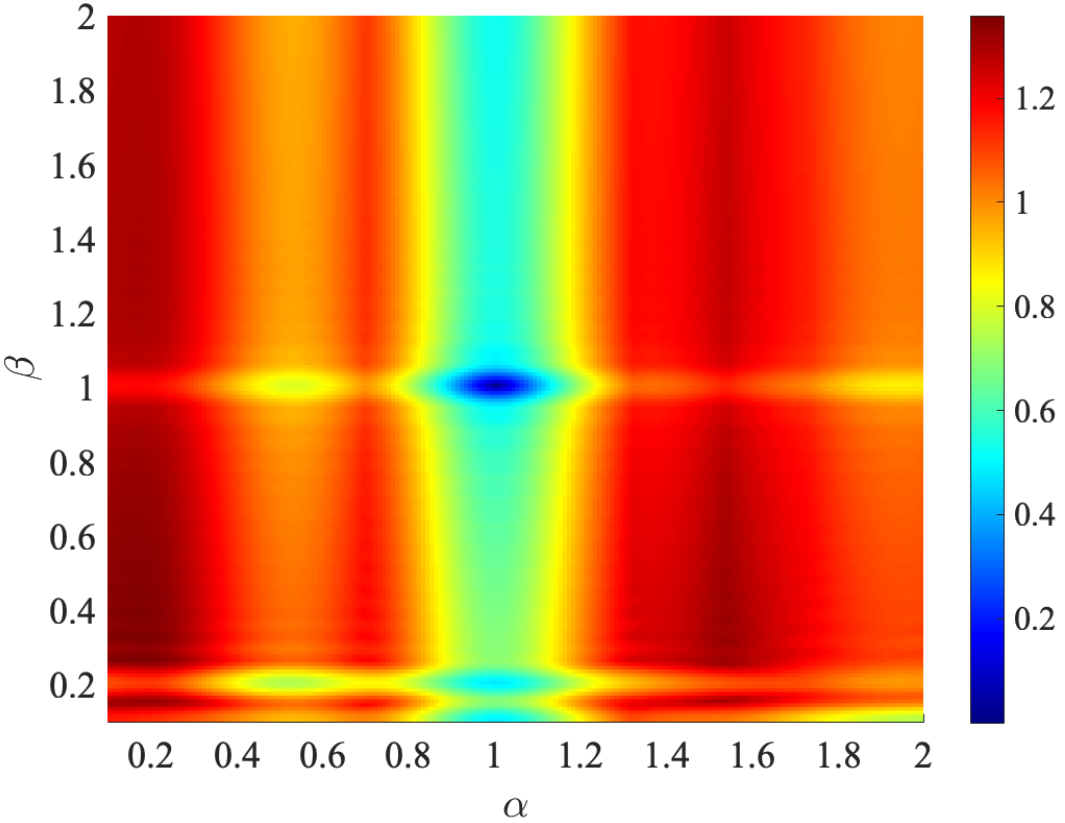}}
\subfigure(b) { \includegraphics[width=0.3\textwidth]{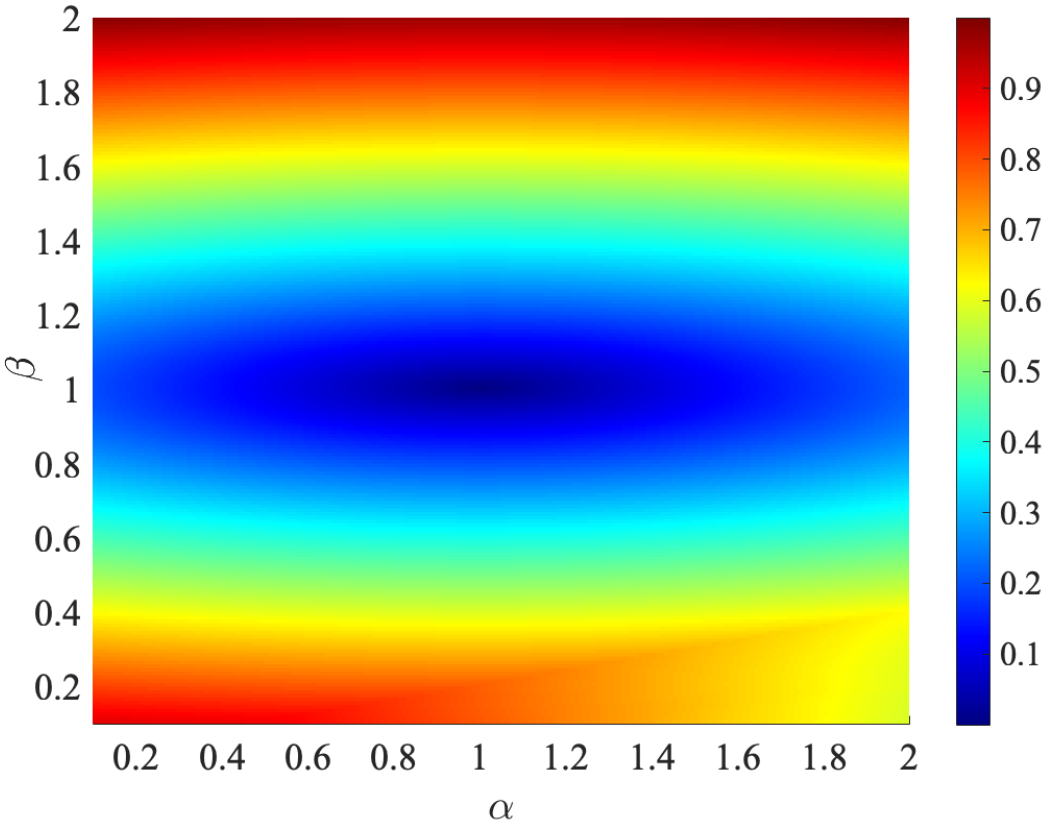}}
\subfigure(c) { \includegraphics[width=0.3\textwidth]{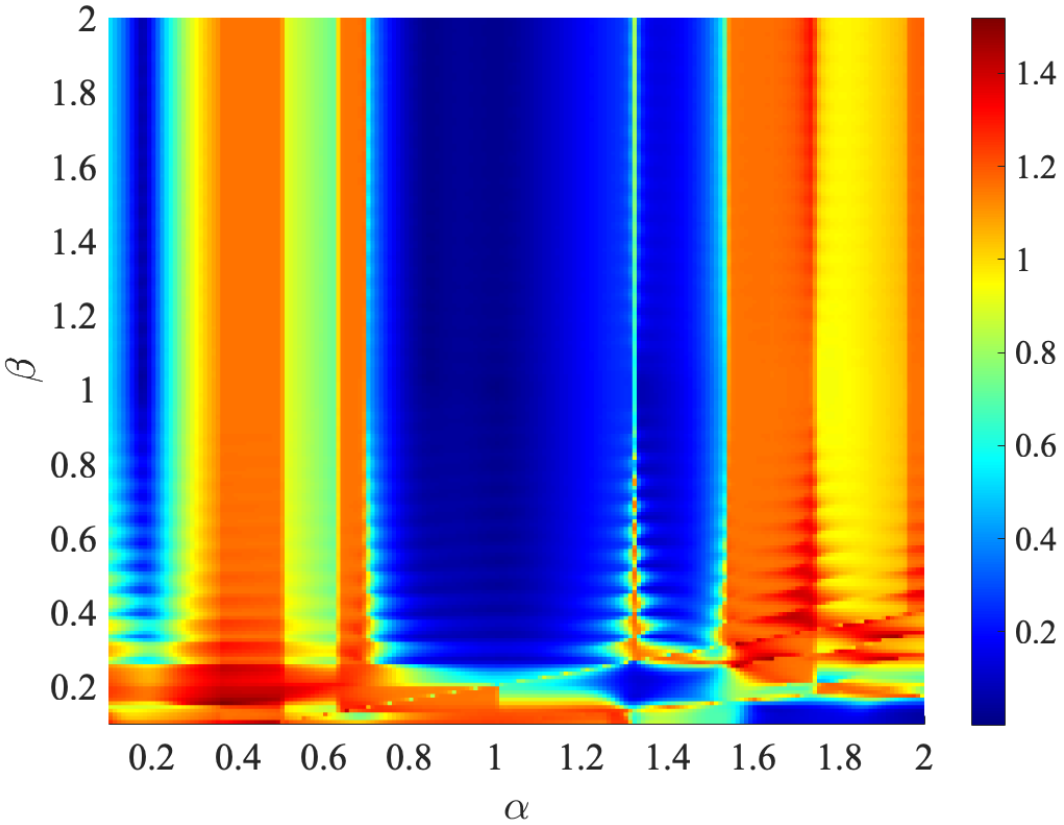}}
\caption{Conjugacy residuals when optimized independenty and comparison between optimal solutions $C_{r_1}$ and $C_{r_2}$; (a) $r_1(C_{r_1})$, (b) $r_2(C_{r_2})$, (c) $\left\Vert C_{r_1} - C_{r_2} \right\Vert_F/\left\Vert C_{r_2} \right\Vert_F$}
\label{r1_r2_alpha_beta}       
\end{figure*}

\begin{figure*}
\subfigure(a) { \includegraphics[width=0.45\textwidth]{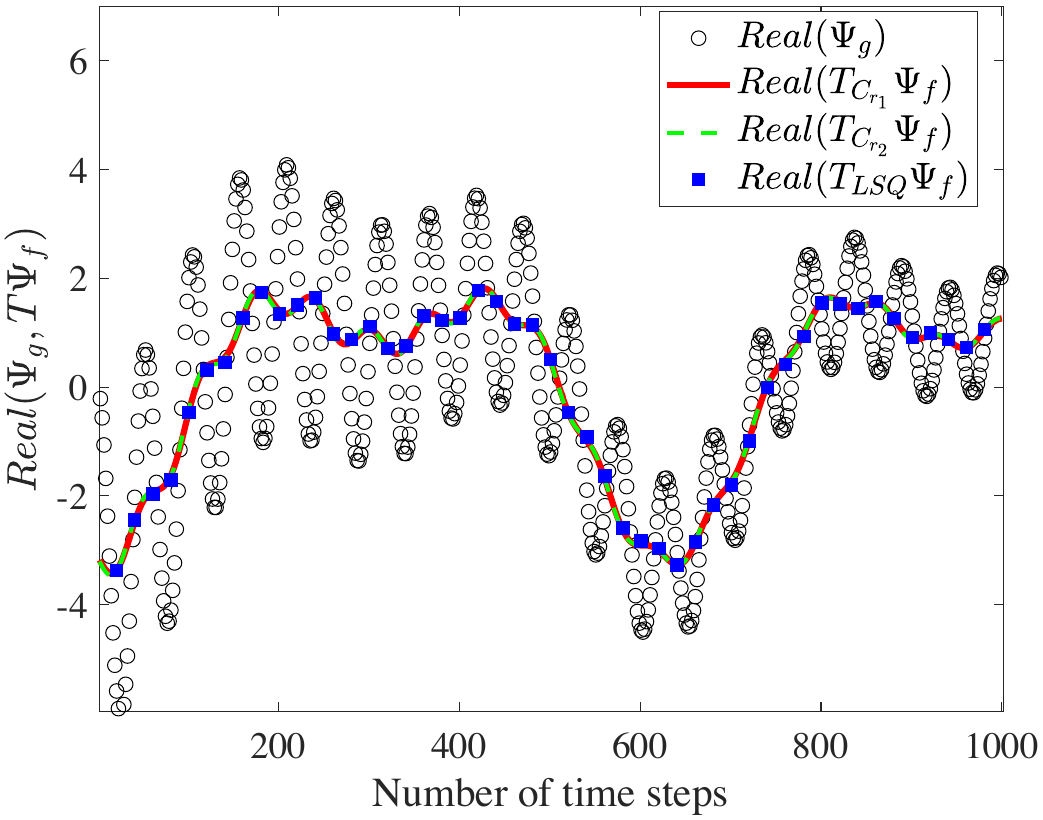}}
\subfigure(b) { \includegraphics[width=0.45\textwidth]{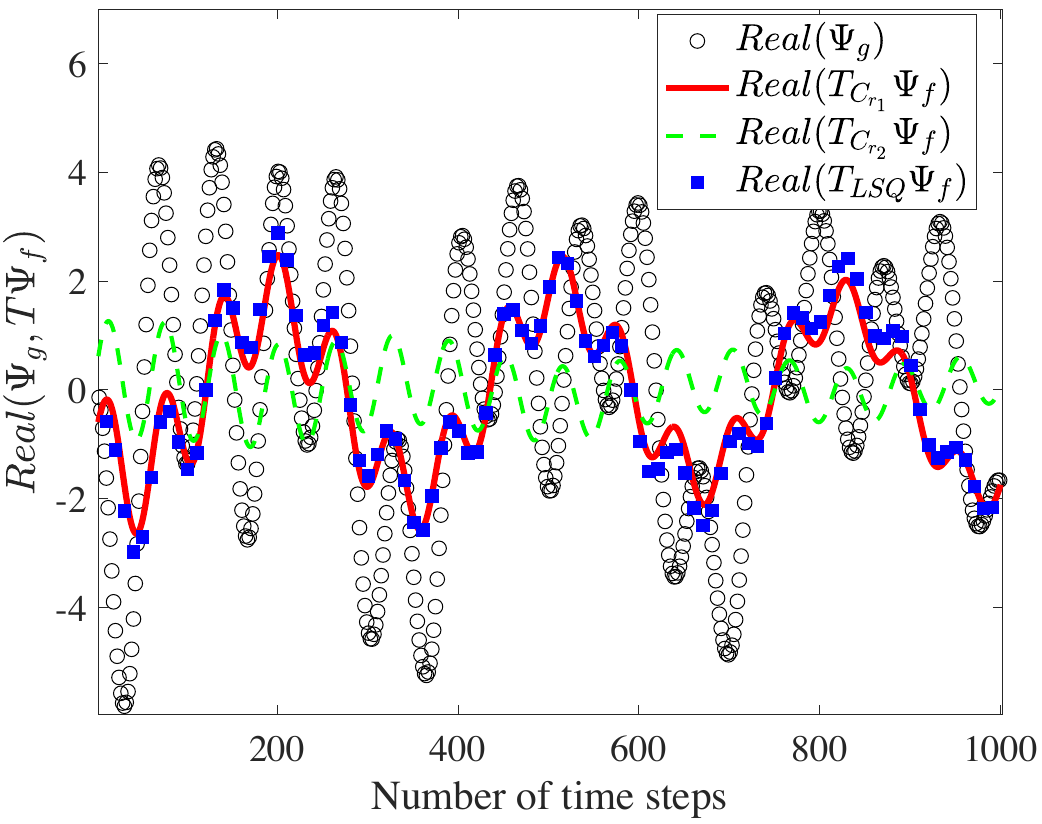}}
\caption{Comparisons of the $y_1$ component of $\Psi_g$, and transformed trajectory $T\Psi_f$; (a) $\alpha = 1,\beta=1.2$, (b) $\alpha = 1.85,\beta= 0.94$}
\label{psi_far}       
\end{figure*}
Additional insight can be gained through visualization of the trajectories. For instance consider the two cases illustrated in Fig. \ref{psi_far}.  The operator residual $r_2(C_{r_2})$ for both cases in Figs. \ref{psi_far}(a) and \ref{psi_far}(b) is 0.20.  For reference, the least squares solution $T_{LSQ} = \Psi_g\Psi_f^{\dagger}$ which is the best possible solution in $GL(N_{\Psi})$ from a trajectory geometry perspective,is provided. In the case shown in Figs. \ref{psi_far}(a), the optimal operator based solution $C_{r_2}$and optimal trajectory geometry solution $C_{r_1}$ lead to nearly identical results.  In this case, it would not matter whether conclusions about dissimilarity are made by using $r_1$ or $r_2$. In contrast, Fig. \ref{psi_far}(b) illustrates a case in which the optimal solutions for $r_1$ and $r_2$ are different.  The trajectory residuals are $r_1(C_{r_1}) = 1.03$ and $r_1(C_{r_2}) = 1.28$ for this case whereas the residuals for the case depicted in Fig. \ref{psi_far}(a) are $r_1(C_{r_1}) = r_1(C_{r_2}) = 0.55$.  So while accounting for $r_2$ alone may lead one to conclude that the cases $\alpha = 1,\beta=1.2$ and $\alpha = 1.85,\beta= 0.94$ are equally dissimilar to the case $\alpha,\beta = 1$, accounting for both $r_1$ and $r_2$ instead indicates that the system corresponding to $\alpha = 1.85,\beta= 0.94$ is more dissimilar to the $\alpha,\beta = 1$ system compared to $\alpha = 1,\beta=1.2$.  This is confirmed in the deviation from conjugacy pseudmoetrics which are $d_{min} = d_{avg} = d_{max}=0.58$ for the $\alpha = 1,\beta=1.2$ case and $d_{min}=1.05$, $d_{avg}=1.24$, and $d_{max}=1.31$ for the $\alpha = 1.85,\beta=0.94$ case. 

\begin{figure*}
\subfigure(a) { \includegraphics[width=0.45\textwidth]{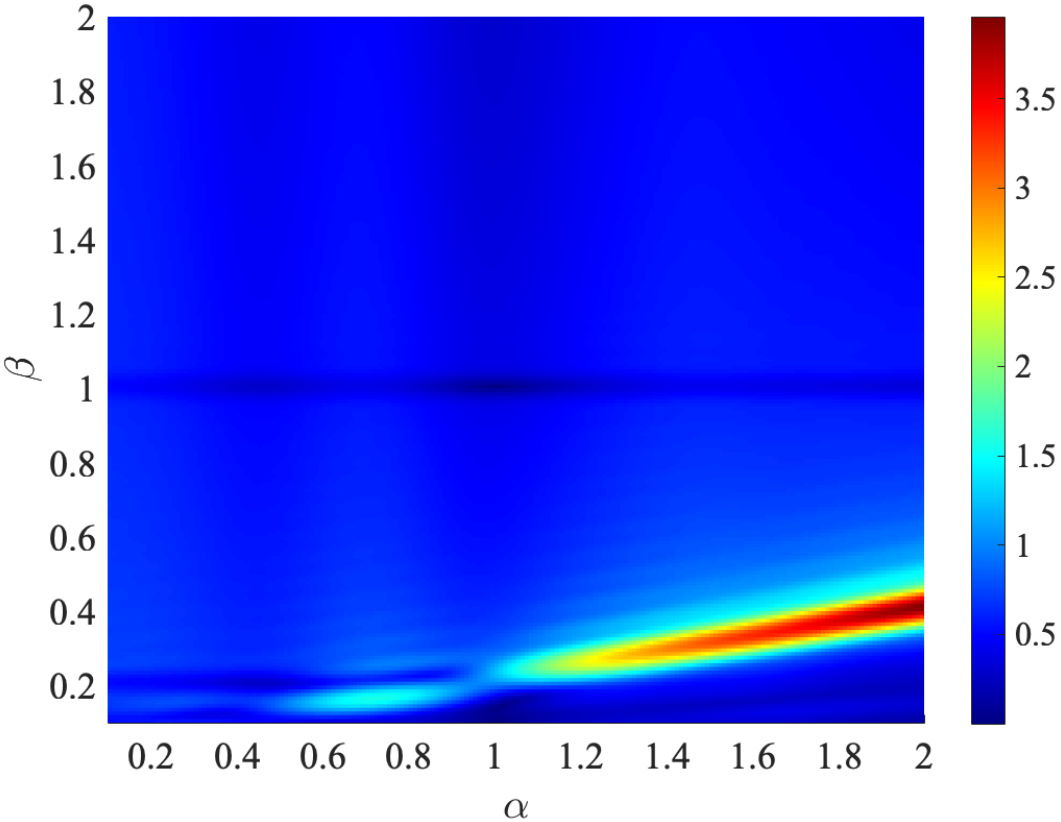}}
\subfigure(b) { \includegraphics[width=0.45\textwidth]{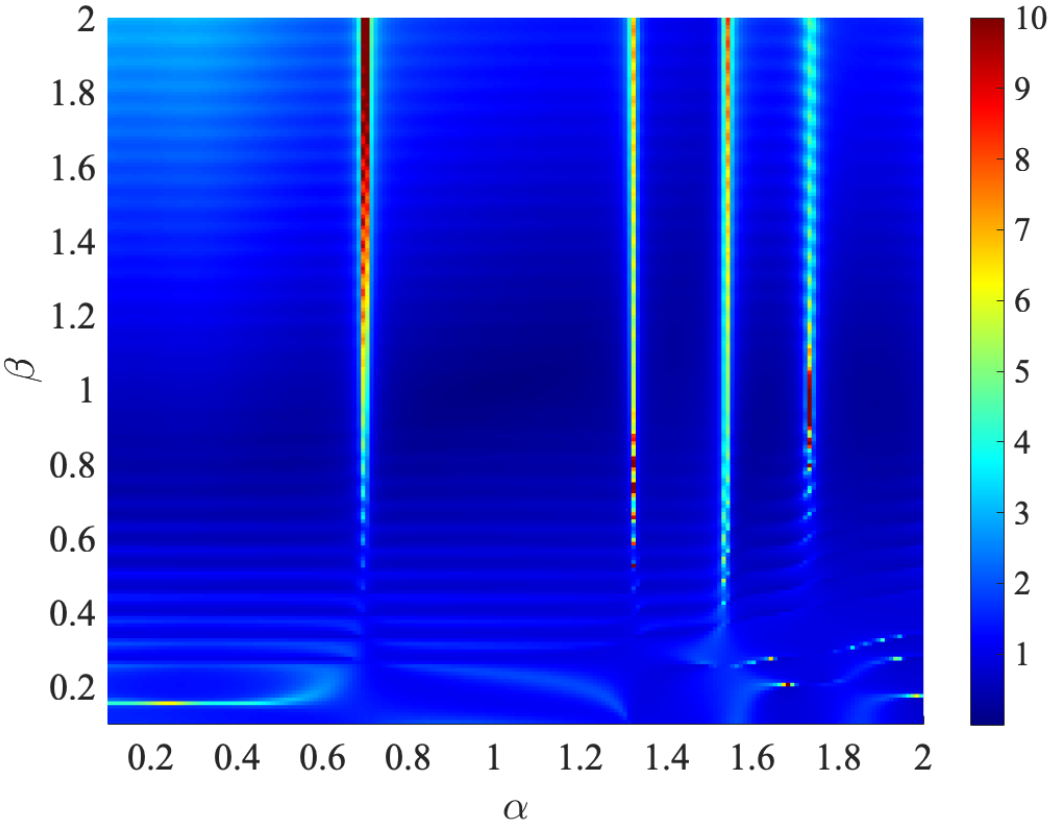}}
\caption{$\Psi$-space residuals with $T_{LSQ}$ varied over $\alpha$ and $\beta$; (a) residual given by Eq. \ref{traj_conj}, (b) residual given by Eq. \ref{sys_conj}}
\label{lsq_os_alpha_beta}       
\end{figure*}
As a an efficient alternative to computing the deviation from conjugacy pseudometrics one may be tempted to consider the least squares solution given by $T_{LSQ} = \Psi_g\Psi_f^{\dagger}$. The $\Psi$-space residuals associated with Eqs. \ref{sys_conj} and \ref{traj_conj} are shown in Fig. \ref{lsq_os_alpha_beta}. While both residuals return zero at conjugacy ($1e^{-11}$ and $2e^{-13}$), they are susceptible to large variations for small changes in $\alpha,\beta$. This can be seen from the large regions of relatively low residual broken up by streaks of high residual.  In contrast, deviations from conjugacy shown in Fig. \ref{dplots} capture more gradual variation with $\alpha,\beta$, and thus the expected result that a system with $\alpha = \alpha_0$, $\beta = \beta_0$ is similar to a system with $\alpha = \alpha_0+\epsilon$, $\beta = \beta_0+\epsilon$, where $\epsilon$ is small. 

It is interesting to note that while neither transformation results in particularly accurate representations of $\Psi_g$ in Figs. \ref{psi_far}(a) and \ref{psi_far}(b), the transformation $T_{C_{r_1}}$ is quite close to the best possible solution (from a trajectory perspective) in $GL(N_{\Psi})$ given by $T_{LSQ}$. In fact trajectory error given by Eq. \ref{traj_conj} for $T_{C_{r_1}}$ is only $0.01 \%$ higher than the least squares case. However, as a tradeoff the operator error in $\Psi$-space given by Eq. \ref{sys_conj} for $T_{LSQ}$ is $48 \%$ higher than the $T_{C_{r_1}}$ case. Since the least squares solution focuses only on the trajectory geometry, it has no direct mechanism to account for the dynamics that generated the trajectory while the deviation from conjugacy pseudometrics do. As a result, least squares solutions in $\Psi$-space are susceptible to obtaining small decreases in trajectory error at the expense of large increases in operator error.


\subsection{Engineering example}

As an engineering example, morphological computation \cite{pfeifer_bongard} of a single degree of freedom mechanical actuator driven by a controller to achieve stable hopping is considered.  Morphological computation is a measure of how much the body can offload the required information processing of the controller.  Systems with high morphological computation will require less information to be processed by a control system in order to achieve a task compared to a system with low morphological computation.  In particular,  actuation and control models representative of biological muscle and DC motor actuators are compared \cite{MCt,haeufle_models}. The quantification of state-dependent morphological computation from \cite{MCt} is used because the dynamics were shown to provide more insight than time averaged quantities.  Since morphological computation is a nonlinear state-dependent quantity it can be treated as a Koopman observable.

Using the simple model from  \cite{haeufle_hopping}, hopping is modeled as
\begin{equation}
\label{hopping}
m \ddot{y} = -mg +
\left\{
\begin{matrix}
0 & & y > l_0 \; \mathrm{flight \; phase} \\
F_L(y,\dot{y},u) & &  y \le l_0 \; \mathrm{ground \; contact}
\end{matrix}
\right.
\end{equation} 
where $y$ is the vertical displacement, $m$ is the system mass,  $F_L$ is the representative leg force,  $u$ is the control input, and $l_0$ is the leg's uncontracted length. For the nonlinear biological muscle system, $F_L$ has a Hill-type force-displacement and force-velocity relationship,  with control input $u$ modeled as a simple force-feedback model representative of neural muscle stimulation.  For the DC motor driven actuator, $F_L$ is modeled as a linear function of the electrical characteristics of the motor. The control input in the DC motor case is obtained via proportional-derivative control to ensure that the DC motor $y,\dot{y}$ trajectories follow those of the nonlinear muscle. A linearized muscle system with linearized force-velocity relationship and the same force feedback control law as the nonlinear muscle is also considered. The latter two systems can be viewed as man-made or bio-inspired attempts to replicate nonlinear muscle, where the DC motor system is an attempt to directly mimic the behavior of the nonlinear biological muscle system while the linearized muscle can be viewed as an attempt to approximate the material properties of the nonlinear muscle. All actuators were tasked with stable hopping at the same height. See \cite{MCt} for full details of the actuator models and parameters. 

State-dependent morphological computation is modeled as \cite{MCt}
\begin{equation}
\label{MCstate}
MC(y,\dot{y},u) = I(w_{n+1};w_n)-I(u;s),
\end{equation}
where $I(w_{n+1};w_n)$ is the mutual information of world states $w$ at consecutive time steps and $I(u;s)$ is the mutual information between actuator and sensor states. Note that mutual information is given by
\begin{equation}
\label{MI}
I(x,y) = -\sum_x p(x) \log_2 p(x) + \sum_{x,y} p(x,y) \log_2 p(x\vert y)
\end{equation}
and the probability terms $p(x)$ and $p(x,y)$ are computed by discretizing the data into $B_y$ bins and calculating the relative number of occurrences for each value. The world states are given by $w = y + B_y \dot{y} + B_y^2 \ddot{y}$ \cite{MCt}.  For the nonlinear muscle system, the sensor state is $F_L$, while for the DC motor system the sensor states are computed from $y$ and $\dot{y}$.The first term in Eq. \ref{MCstate} is high if the system shows diverse but non-random dynamics while the second term is low if the system's control policy has a low diversity in it's outputs $u$ or there is low correlation between sensor states and actuator states. Therefore Eq. \ref{MCstate} is high if the system can produce complex behavior based on a controller with low complexity \cite{MCt}.

The primary Koopman observables are chosen to be \\$\Psi_{\mathrm{primary},n} = \left[\begin{matrix}  y_n & \dot{y}_n & u_n & MC_n)\end{matrix}\right]^T$.   The Koopman operators were calculated from training data corresponding to $N_{\Delta t} = 2500$ which is aproximately 10 periods of hopping. An additional 1000 time-steps of test data was used to verify the accuracy of the trained models.

The data-driven Koopman predictions of $\Psi_{\mathrm{primary}}$ are compared to the 1000 time-steps of test data in Figs. \ref{koop_nlm} - \ref{koop_lm}.  The data-driven Koopman models accurately capture the system observables, including the highly nonlinear morphological computation observable. Since the controller for each actuator was set to drive the system to hop at the same height, $y$ and $\dot{y}$ are similar for all models. However, the control signals needed to achieve this are quite different as shown in Figs. \ref{koop_nlm}(c), \ref{koop_dc}(c), and \ref{koop_lm}(c). It can also be seen that the dynamics of morphological computation are different for the nonlinear muscle compared to the other systems. Perhaps the most striking difference in Fig. \ref{koop_nlm}(d) compared to Figs. \ref{koop_dc}(d) and \ref{koop_lm}(d) is that morpholigcal computation quickly recovers from low values for the nonlinear muscle. When control is activated, (or turned off) morphological computation sharply decreases since the $I(u;s)$ term in Eq. \ref{MCstate} becomes more important. However, the body dynamics quickly take over in the nonlinear muscle case and morphological computation is able to sharply recover even while control is varying. This dynamical behavior is not present in the `man-made' systems and also would not be captured if only considering the time-average of morphologcal computation.
\begin{figure*}
\subfigure(a) {\includegraphics[width=0.48\textwidth]{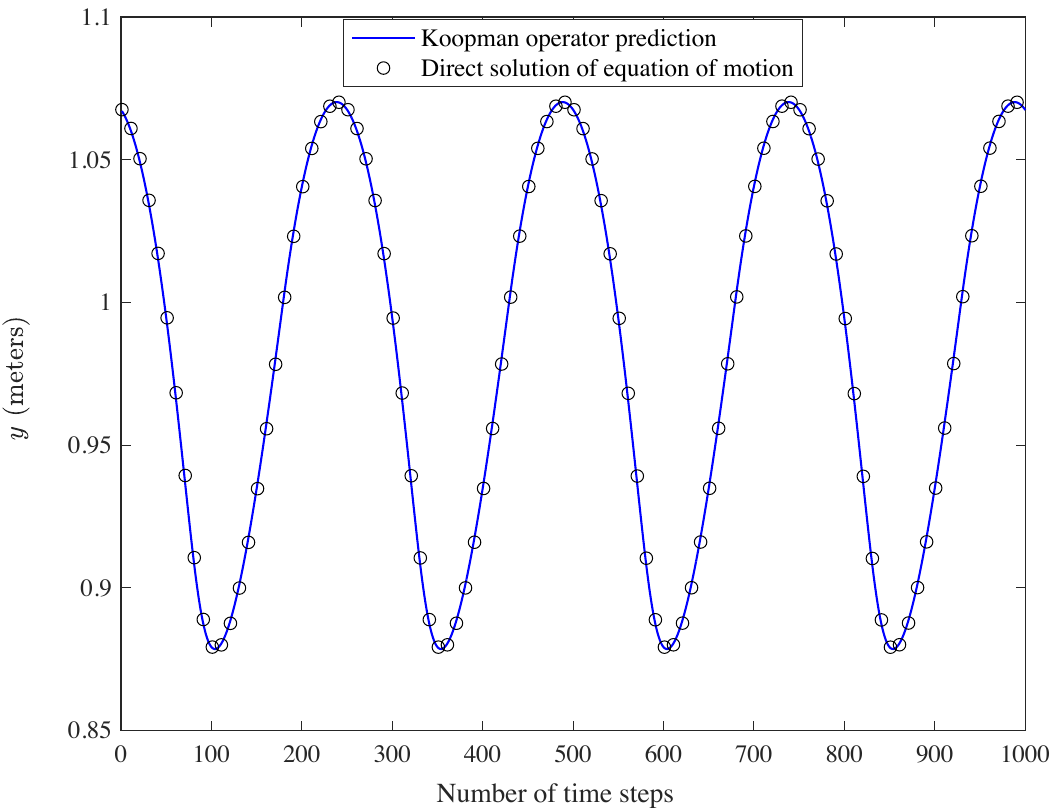}}
\subfigure(b) {\includegraphics[width=0.48\textwidth]{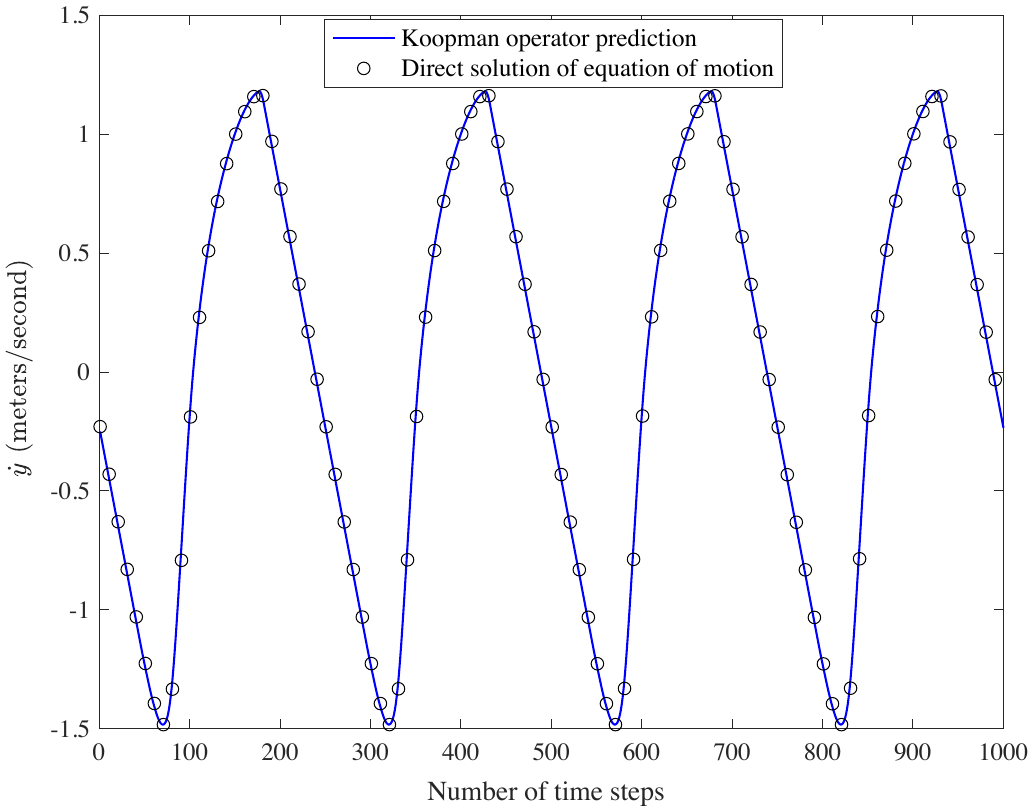}}
\subfigure(c) {\includegraphics[width=0.48\textwidth]{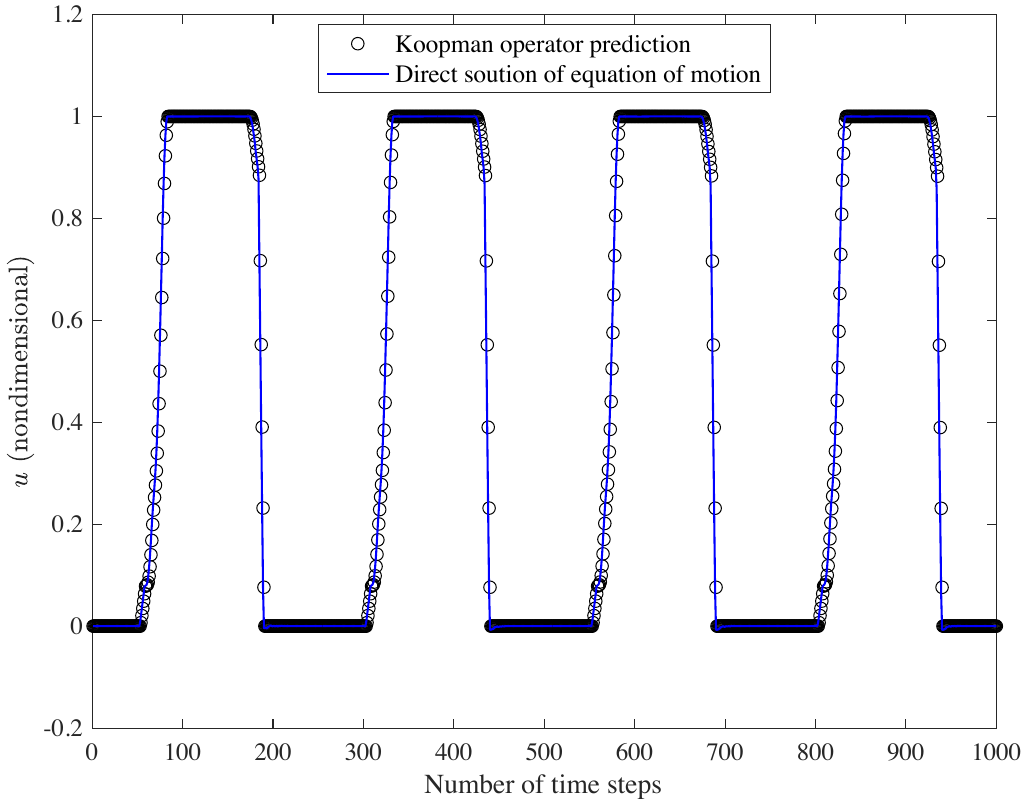}}
\subfigure(d) {\includegraphics[width=0.48\textwidth]{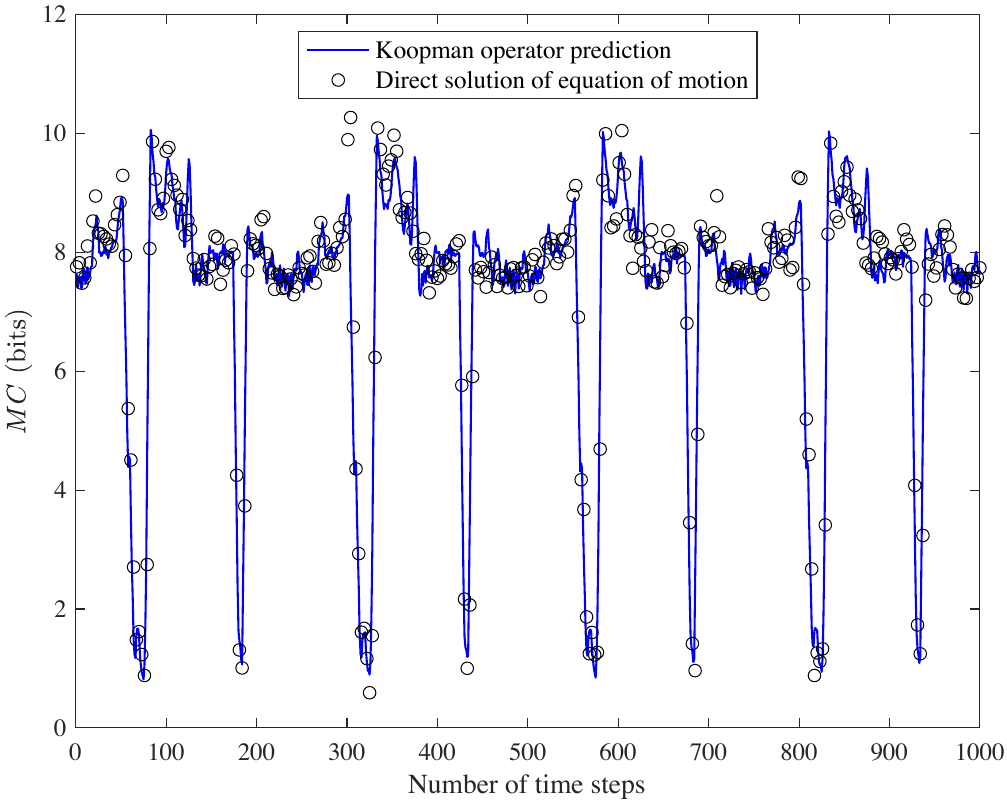}}
\caption{Koopman operator predictions vs direct solutions to equations of motion for the nonlinear muscle actuator}
\label{koop_nlm}       
\end{figure*}

\begin{figure*}
\subfigure(a) {\includegraphics[width=0.48\textwidth]{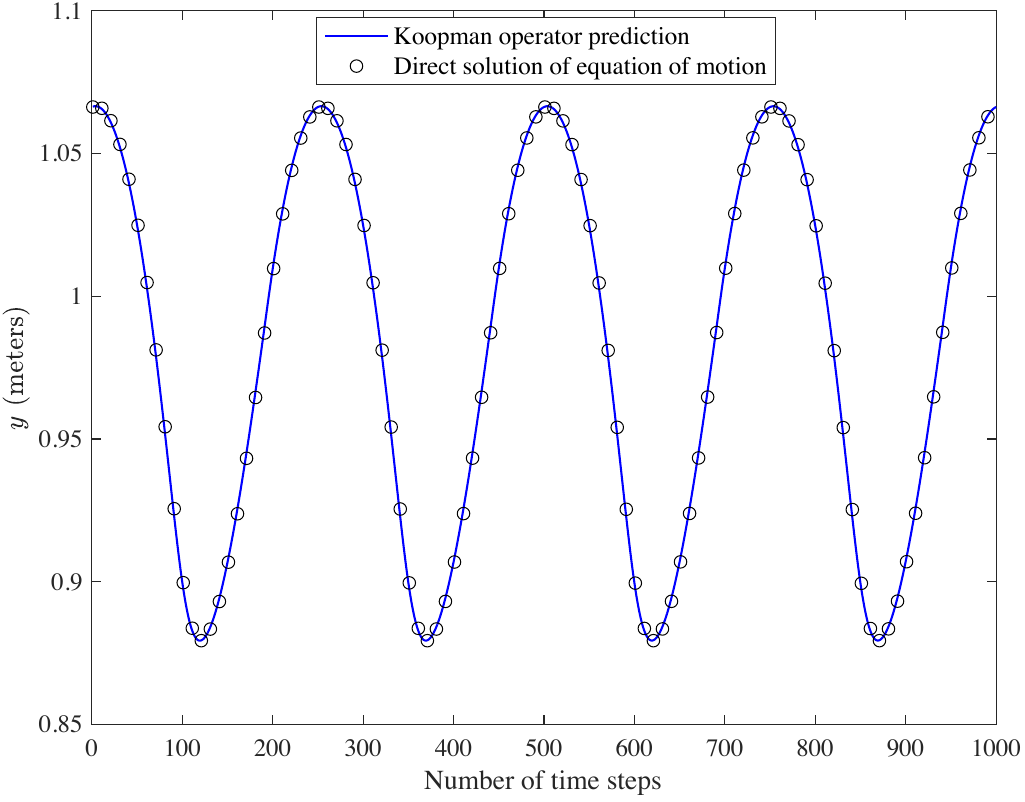}}
\subfigure(b) {\includegraphics[width=0.48\textwidth]{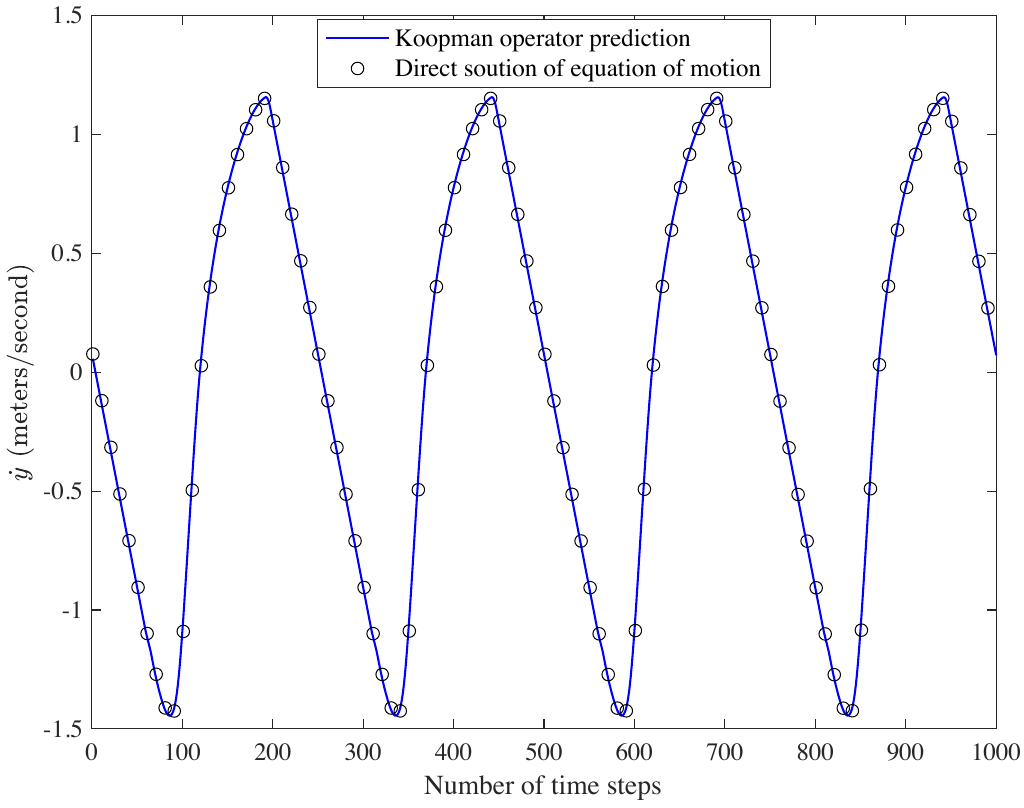}}
\subfigure(c) {\includegraphics[width=0.48\textwidth]{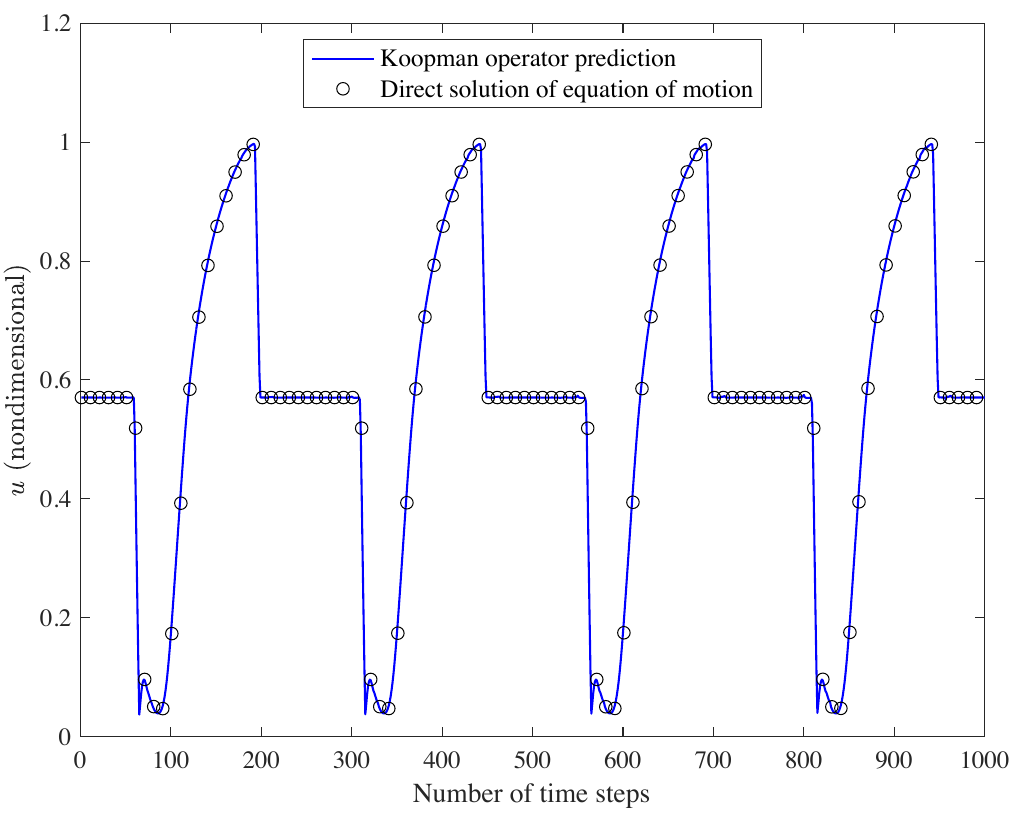}}
\subfigure(d) {\includegraphics[width=0.48\textwidth]{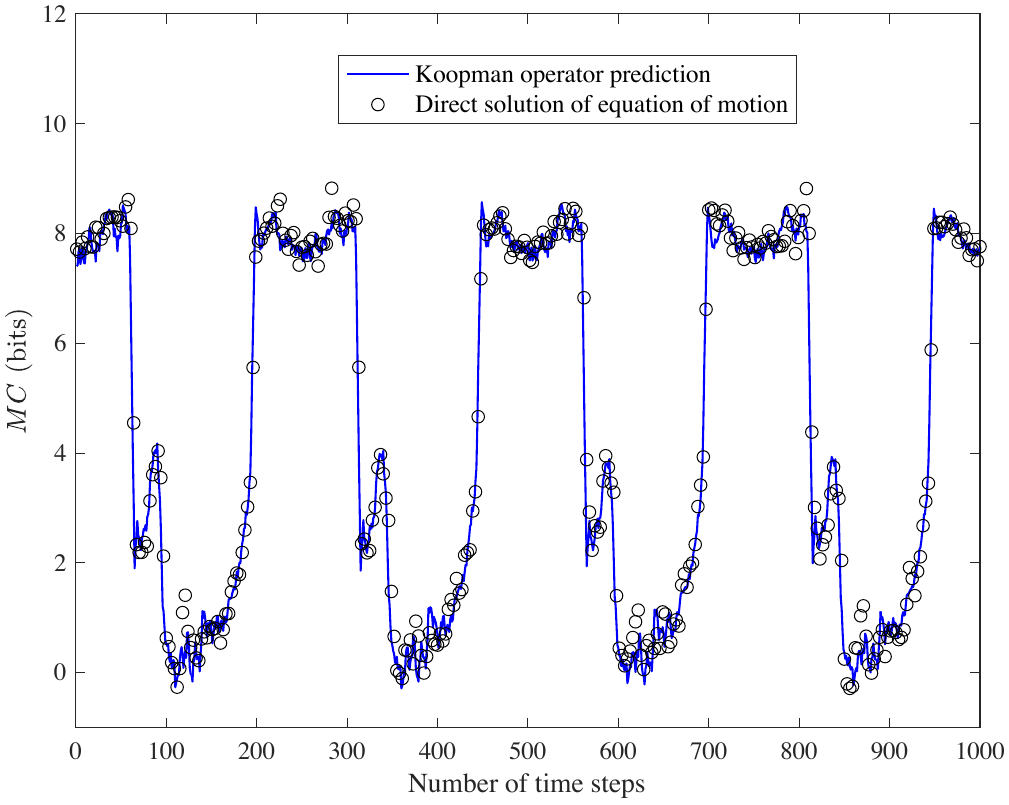}}
\caption{Koopman operator predictions vs direct solutions to equations of motion for the DC motor actuator}
\label{koop_dc}       
\end{figure*}

\begin{figure*}
\subfigure(a) {\includegraphics[width=0.48\textwidth]{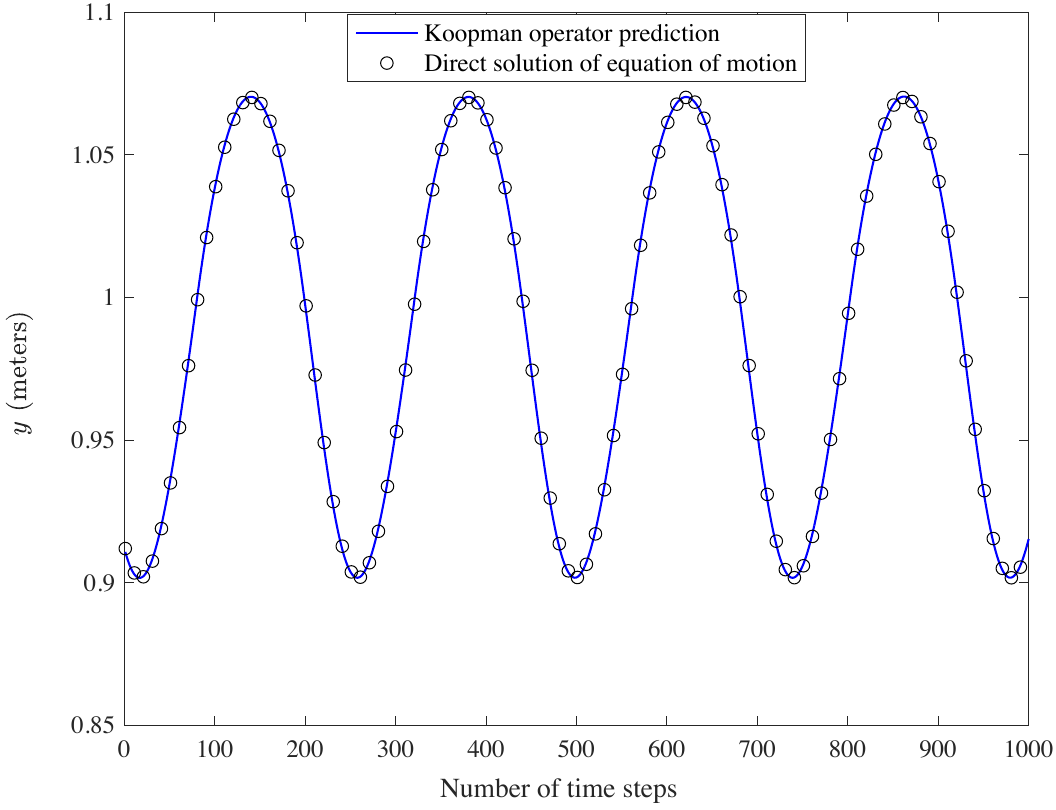}}
\subfigure(b) {\includegraphics[width=0.48\textwidth]{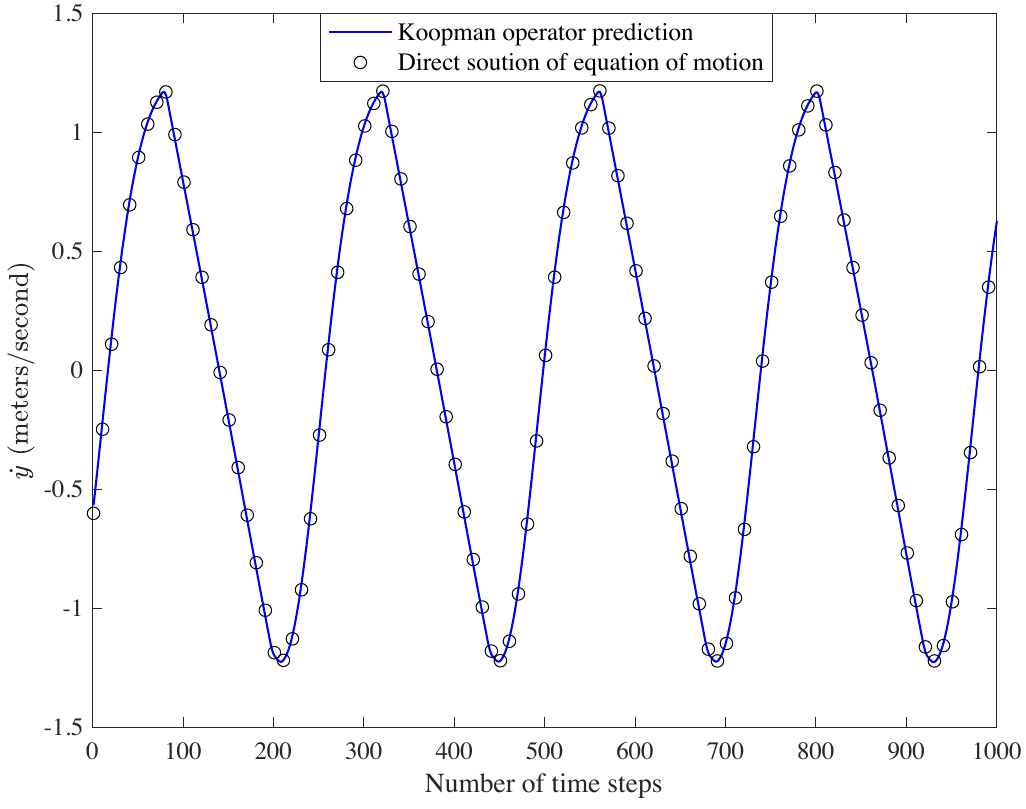}}
\subfigure(c) {\includegraphics[width=0.48\textwidth]{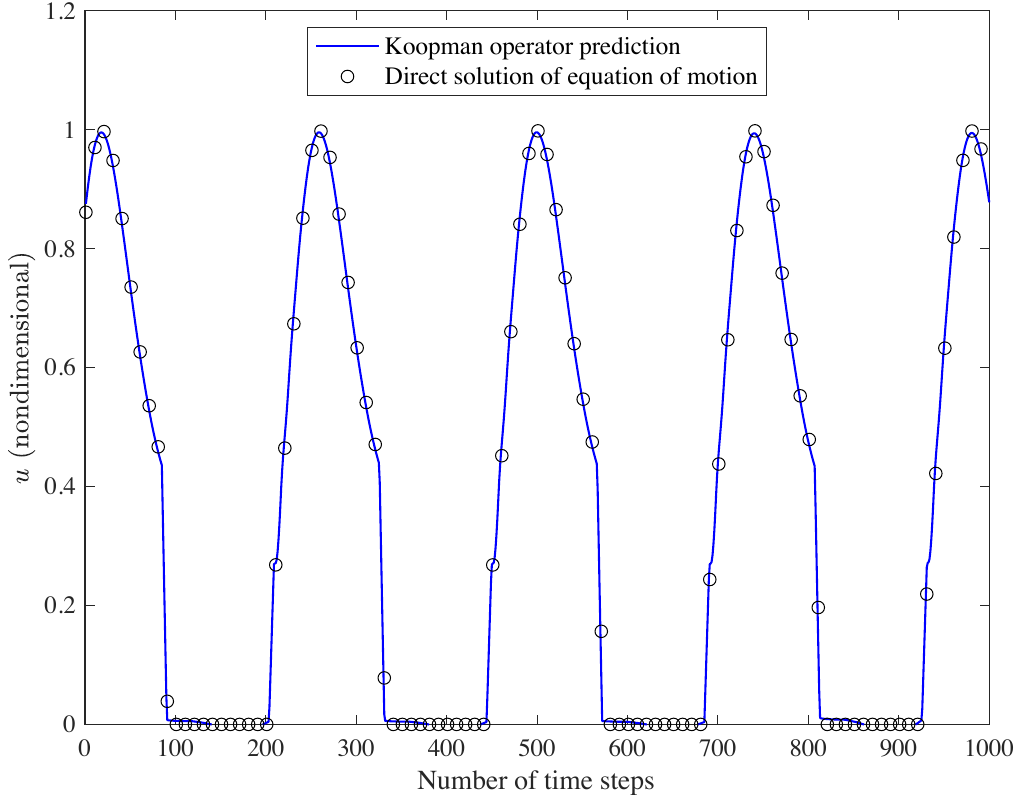}}
\subfigure(d) {\includegraphics[width=0.48\textwidth]{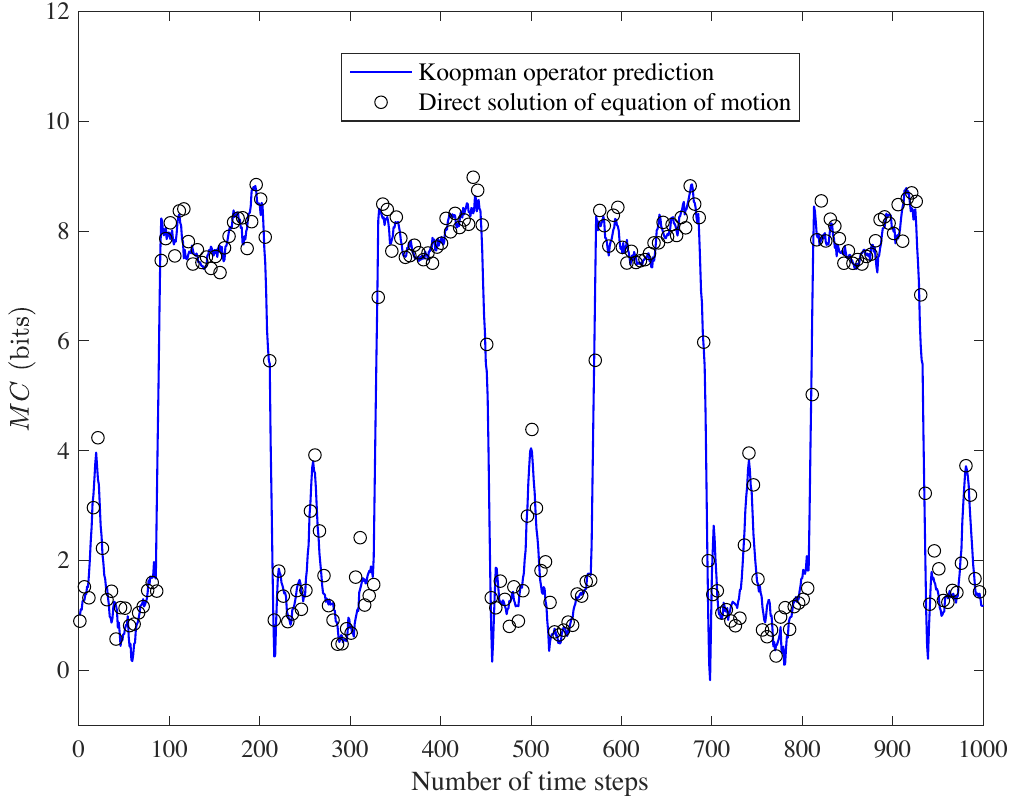}}
\caption{Koopman operator predictions vs direct solutions to equations of motion for the linearized muscle actuator}
\label{koop_lm}       
\end{figure*}

\begin{table*}
\caption{Deviations from conjugacy and conjugacy residuals; NLM - nonlinear muscle, LM - linearized muscle, DC - DC motor.}
\label{tab_dr}       
\centering
\begin{tabular}{lllllllll}
\hline\noalign{\smallskip}
Systems & $d_{min}$ & $d_{avg}$ & $d_{max}$ & $r_1(C_{r_1})$ &  $r_2(C_{r_1})$ & $r_1(C_{r_2})$ &  $r_2(C_{r_2})$\\
\noalign{\smallskip}\hline\noalign{\smallskip}
NLM/LM & 0.91 & 1.06 & 1.20 & 0.91 & 0.06 & 1.20 & 0.01 \\
NLM/DC & 0.29 & 0.58 & 0.86 & 0.29 &  0.05 & 0.86 & 0.01\\
LM/DC & 0.94 & 1.12 & 1.32 & 0.94 &  0.06 & 1.32 & 0.02\\
\noalign{\smallskip}\hline
\end{tabular}
\end{table*}

The deviation from conjugacy pseudometrics and associated conjugacy residuals are provided in Table \ref{tab_dr}.  The pseudometrics are calculated using normalized residuals from Eqs. \ref{r1scaled} and \ref{r2scaled}, with the nonlinear muscle actuator serving as the reference system. The deviations from conjugacy quantify the intuitive notion that all of the dynamical systems are quite different.  However,  the finding that the DC motor attempt to approximate biological muscle is closer than the linearized muscle approximation for all pseudometrics is interesting and perhaps not intuitive.  Furthermore, it is interesting to note that the linearized muscle and DC motor systems are very different from each other with $d_{avg}=1.23$ even though their time-averaged morphological computation values are relatively close; 4.76 bits for the linearized muscle and 4.52 bits for the DC motor compared to 7.30 bits for the nonlinear muscle. As a result, one may conclude based on mean values of morphological computation that the linearized muscle and the DC motor actuator are relatively similar approximations of nonlinear muscle. However a comparison based on the dynamically consistent pseudometrics capture differences due to \emph{dynamics} of morphological computation that may be lost when comparing systems based on time-averaged morphological computation.  

\begin{figure*}
\subfigure(a) {\includegraphics[width=0.48\textwidth]{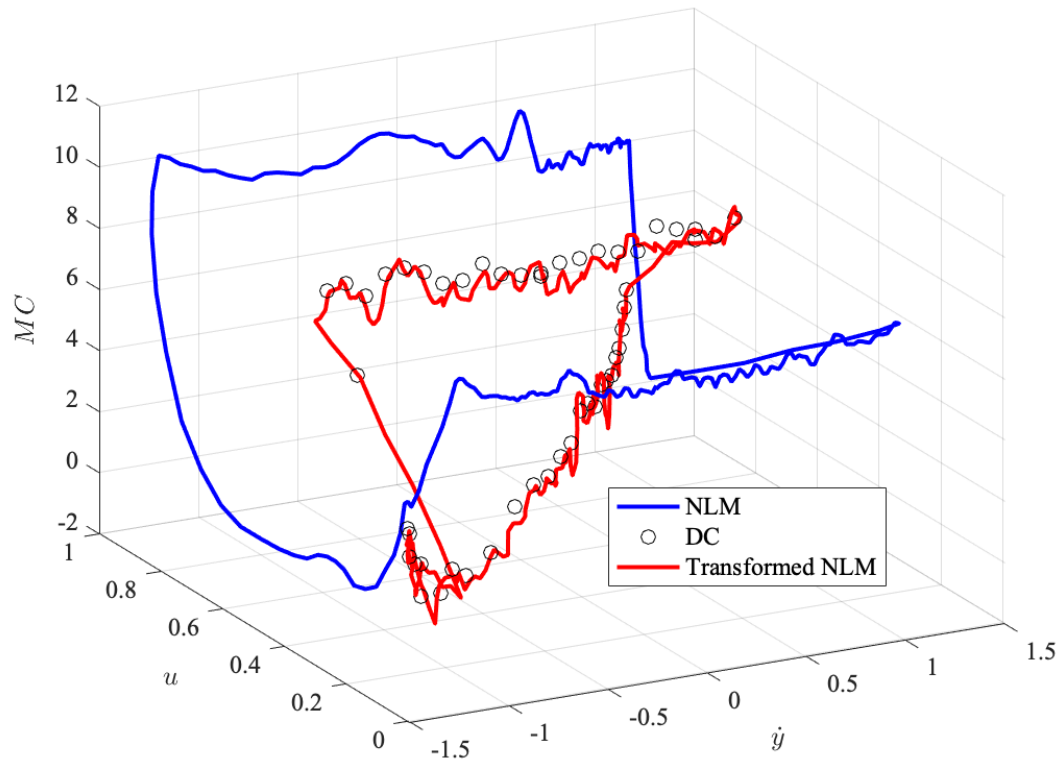}}
\subfigure(b) {\includegraphics[width=0.48\textwidth]{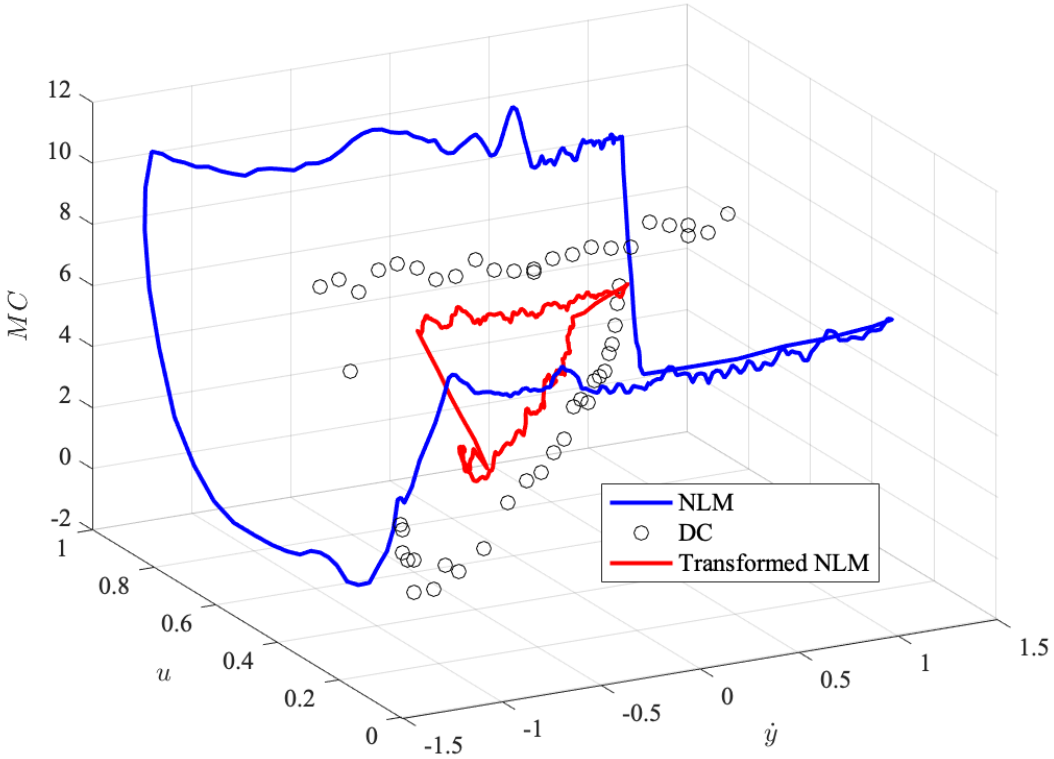}}
\caption{Koopman observable trajectories for components of $\Psi$ corresponding to $\dot{y}$, $u$, and $MC$ for the nonllinear muscle actuator (NLM) and DC motor actuator (DC); (a) $T_C = (\Omega W_g)^{-1}C_{r_{1}} W_f$, (b) $T_C = (\Omega W_g)^{-1}C_{r_{2}} W_f$. }
\label{NLMtoDC}       
\end{figure*}

\begin{figure*}
\subfigure(a) {\includegraphics[width=0.48\textwidth]{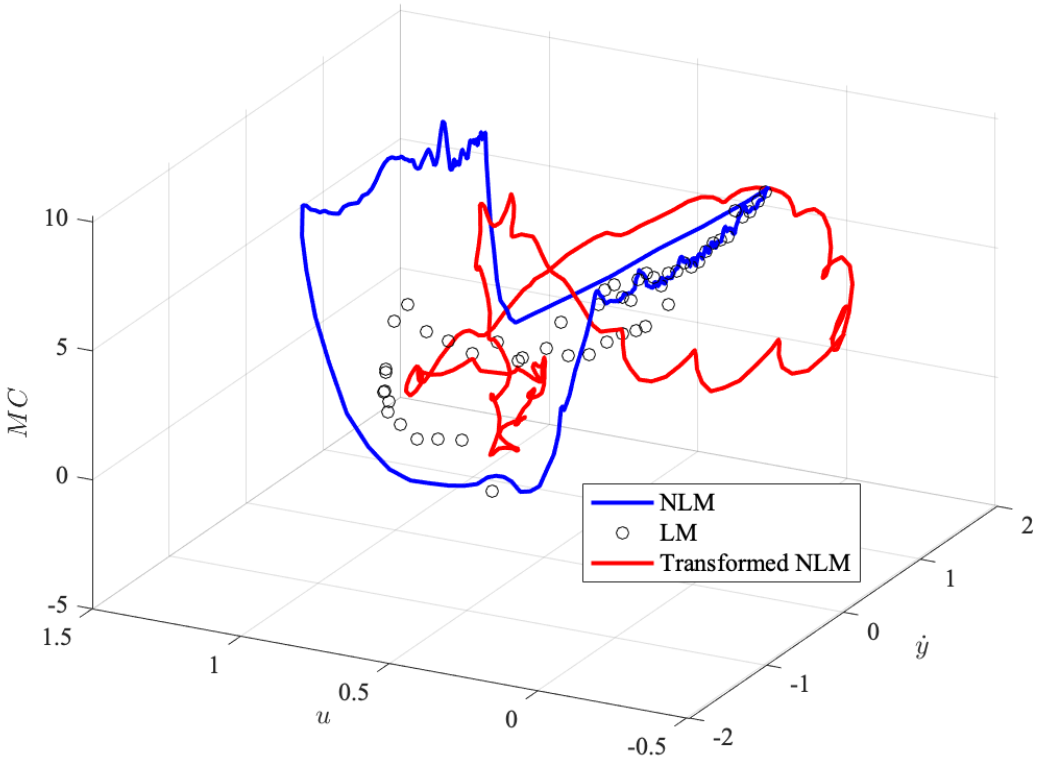}}
\subfigure(b) {\includegraphics[width=0.48\textwidth]{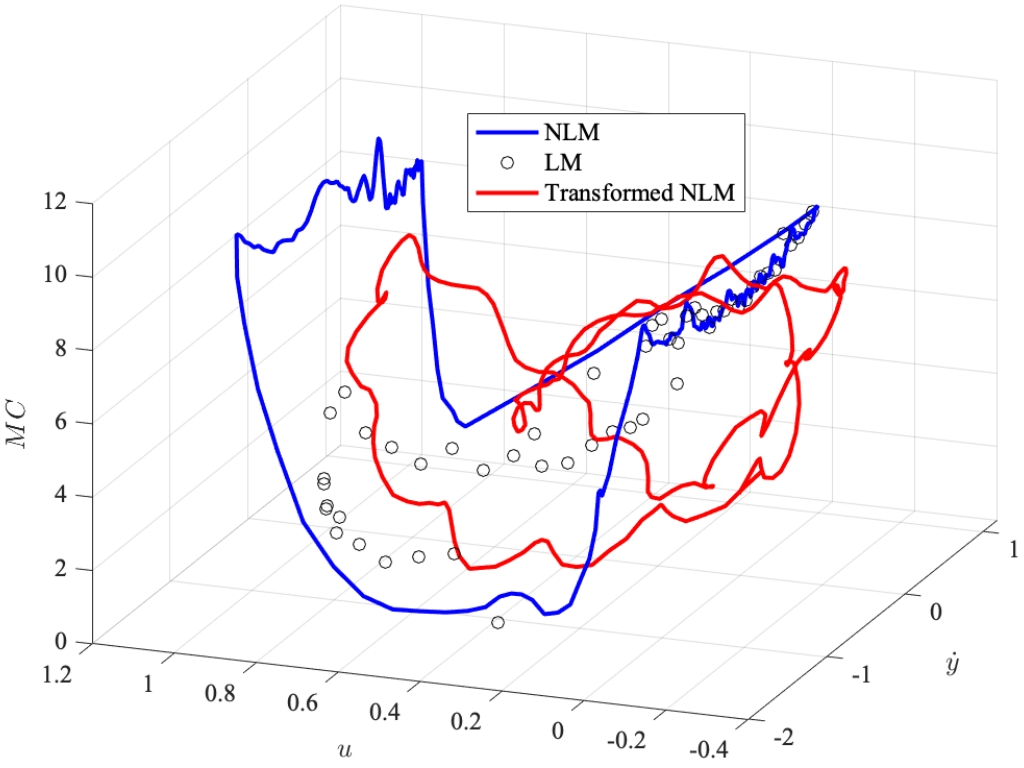}}
\caption{Koopman observable trajectories for components of $\Psi$ corresponding to $\dot{y}$, $u$, and $MC$ for the nonllinear muscle actuator (NLM) and linearized muscle actuator (LM) ; (a) $T_C = (\Omega W_g)^{-1}C_{r_{1}} W_f$, (b) $T_C = (\Omega W_g)^{-1}C_{r_{2}} W_f$. }
\label{NLMtoLM}       
\end{figure*}

Table \ref{tab_dr} shows the tradeoffs between minimizing $r_1$ versus $r_2$ for these systems.  Figures \ref{NLMtoDC} and  \ref{NLMtoLM} provide visual insight to theses numerical values. The DC motor system is moderatley dissimilar to the nonlinear muscle system when considering trajectory geometry as indicated by $r_1(C_{r_1})=0.29$. Optimizing for $r_2$ leads to relatively modest improvements from $r_2(C_{r_1})=0.05$ to $r_2(C_{r_2})=0.01$. However, this comes at the cost of increasing $r_1$ to 0.86. This is illustrated in Fig. \ref{NLMtoDC} where the transformation based on optimizing for $r_1$ shows that the transformed nonlinear muscle system resembles the $\Psi_{\mathrm{primary}}$ geometry of the DC motor system.  Interestingly, the transformed solution in Fig. \ref{NLMtoDC}(b) corresponding to minimizing $r_2$ resembles the shape of the DC motor $\Psi_{\mathrm{primary}}$ trajectories, but does not match the scale. For comparison, Fig. \ref{NLMtoLM} shows that the nonlinear muscle system does not resemble the linearized muscle system when transformed, which tracks with the higher values reported in Table \ref{tab_dr}.

\section{Conclusions}
\label{conc}

Computationally efficient solutions for pseudometrics quantifying deviation from topological conjugacy between dynamical systems were presented.  The theoretical justification for computing such pseudometrics in Koopman eigenfunction space rather than observable space was provided. This simiplification, along with Pareto optimality considerations led to the conclusion that optimal transformations for comparing systems could be obtained from the group of unitary matrices.  It was shown that the pseudometrics could be computed in polynomial time. The resulting practical advantages in terms of efficiency, scalability, and theoretical consistency were discussed.

Results demonstrated the effectiveness of the deviation from conjugacy pseudometrics. A simple benchmarking problem was used to show that deviation from conjugacy recover a value of zero at topological conjugacy even when the underlying transformation between systems is non-unitary.  Furthermore, the benefit of considering a Pareto optimal context for deviation from conjugacy when trajectory geometry is important was illustrated by showing that considering spectral operator error alone may lead to conclusions on dissimilarity that are not reflected when comparing trajectories.  In contrast, the deviation from conjugacy pseudometrics accomodate both measures of dissimilarity in a consistent formalism and avoid such pitfalls.  Finally, an engineering example was presented in which biological nonlinear muscle based mechanical actuation systems were compared to DC motor and linearized muscle systems. The pseudometrics were able to quantify dissimilary between the systems based on the dynamics of morphological computation (a Koopman observable). It was also shown that while the DC motor and linearized muscle have relatively similar time-averaged morphological computation, they do not have similar morphological computation dynamics.


\section*{Acknowledgements}
I am thankful to Igor Mezi{\'c} for the useful discussion on this topic.
\bibliographystyle{spmpsci}      
\bibliography{metrics_paper_arxiv.bib}   

%
%



\end{document}